\documentclass[11pt]{article}

\PassOptionsToPackage{nameinlink,noabbrev}{cleveref}
\usepackage[T1]{fontenc}
\usepackage[margin=0.92in]{geometry}

\usepackage{amsmath,amssymb,amsthm}
\usepackage{graphicx}
\usepackage{xcolor}
\usepackage[colorlinks=true, allcolors=blue]{hyperref}

\usepackage{algorithm} 
\usepackage{algorithmic}  
\usepackage[algo2e]{algorithm2e} 
\usepackage{natbib}
\newtheorem{thm}{Theorem}[section]
\newtheorem{theorem}{Theorem}[section]

\newtheorem{proposition}[thm]{Proposition}

\newtheorem{lemma}[thm]{Lemma}

\newcommand{\one}{\mathbf 1}

\title{On Minimax Optimality and Uniqueness of Fixed-Step First-Order Methods for Smooth Convex Optimization}

\author{
Benjamin Grimmer\thanks{All authors contributed equally and are listed in alphabetical order by surname.}\\
\small Johns Hopkins University\\
\small Massachusetts Institute of Technology
\and
Sunghyeon Jo\footnotemark[1]\\
\small Georgia Institute of Technology
\and
Chanwoo Park\footnotemark[1]\\
\small Massachusetts Institute of Technology
}
\date{}

\begin{document}
\maketitle

\begin{abstract}
    This paper considers the design of optimal fixed-step first-order methods for high-dimensional minimization of $L$-smooth convex functions. For optimizing worst-case performance measured via suboptimality of the final function value (relative to the initial squared distance to a minimizer), we provide an algebraic proof of the optimality of the optimized gradient method (OGM) and establish its uniqueness among all fixed-step first-order methods. For the alternative measure of final squared gradient norm (relative to initial suboptimality), we prove the OGM-G method is optimal and uniquely so among fixed-step first-order methods. Finally, for the setting measuring the final squared gradient norm (relative to the initial squared distance to a minimizer), we show the recently proposed Lemniscate method is optimal and uniquely so. Our proofs rely on algebraic reductions for lower bound arguments rather than traditional information-theoretic bounds, which were previously only able to establish OGM's optimality but not uniqueness.
\end{abstract}

\section{Introduction}
We consider the setting of smooth convex optimization via fixed-step first-order methods. Namely, we consider optimization problems, given an initialization $x_0\in\mathbb{R}^d$, of the form
\[
    f_\star = \min_{x\in\mathbb{R}^d} f(x),
\]
where $f\colon \mathbb{R}^d\rightarrow \mathbb{R}$ is $L$-smooth (meaning it has an $L$-Lipschitz continuous gradient), convex, and attains its minimum at some point $x_\star\in\mathbb{R}^d$. The suboptimality of the function value at $x$ is $f(x)-f_\star$. We consider two models for initial conditions on such problems: we will suppose either the initialization has bounded squared distance to a minimizer $\frac{1}{2}\|x_0-x_\star\|^2\leq R$ or has bounded initial suboptimality $f(x_0)-f_\star \leq R$.

Without loss of generality, we can take $L=R=1$ by simple rescalings. We denote the resulting family of such problem instances in dimension $d$ with bounded squared distance to a minimizer by 
\[ \mathtt{P}_\mathtt{dist}(d) = \{(f,x_0)\mid f\text{ is $1$-smooth, convex, attains its minimum at some $x_\star$ with $\frac{1}{2}\|x_0-x_\star\|^2\leq 1$} \} \]
and the family of problem instances with bounded initial suboptimality by
\[\mathtt{P}_\mathtt{subopt}(d) = \{(f,x_0)\mid f\text{ is $1$-smooth, convex, attains its minimum at some $x_\star$ with $f(x_0) - f_\star \leq 1$} \}.\]

For a given iteration budget $N\geq 1$, we consider the design of $N$-step fixed-step first-order methods, defined by a lower triangular matrix $W\in\mathbb{R}^{(N+1)\times(N+1)}$, iterating for $n=1,\dots,N$
\begin{equation}
    x_n = x_{0} - \frac{1}{L}\sum_{i=0}^{n-1} W_{n,i} g_i \label{eq:fixed-step-definition}
\end{equation}
where $g_i=\nabla f(x_i)$. The coefficients $W_{n,i}$ are prescribed in advance and do not depend on the observed function values or gradients. Here, for ease, we index the columns and rows of $W$ from $0$ to $N$. It is convenient for us to set the diagonal of $W$ equal to one (called a unit lower triangular matrix). In this notation, one then arrives at a gradient descent step from each iterate by
\begin{equation*}
    x_n - \frac{1}{L} g_n = x_{0} - \frac{1}{L}\sum_{i=0}^{n} W_{n,i} g_i
\end{equation*}

We denote the set of all such fixed-step first-order methods with iteration budget $N$, identified by their matrix $W$, by $\mathtt{A}_\mathtt{fixed}(N)$. Among these, we seek to identify the methods with the best worst-case guarantee. We consider two distinct performance measures, namely the final suboptimality $f(x_N)-f_\star$ and the final squared gradient norm $\frac{1}{2}\|\nabla f(x_N)\|^2$. From these, we consider the three central design problems for $N$-step methods in dimension $d$, denoted by
\begin{align}
    r_{D\rightarrow F} &= \inf_{W\in\mathtt{A}_\mathtt{fixed}(N)} \sup_{(f,x_0)\in\mathtt{P}_\mathtt{dist}(d)} f(x_N)-f_\star, \label{eq:disttosubopt-design}\\
    r_{F\rightarrow G} &= \inf_{W\in\mathtt{A}_\mathtt{fixed}(N)} \sup_{(f,x_0)\in\mathtt{P}_\mathtt{subopt}(d)} \frac{1}{2}\|\nabla f(x_N)\|^2, \label{eq:subopttograd-design}\\
    r_{D\rightarrow G} &= \inf_{W\in\mathtt{A}_\mathtt{fixed}(N)} \sup_{(f,x_0)\in\mathtt{P}_\mathtt{dist}(d)} \frac{1}{2}\|\nabla f(x_N)\|^2. \label{eq:disttograd-design}    
\end{align}

The first algorithm design problem above is attained by the optimized gradient method (OGM) as analyzed by~\cite{KF16}; its exact information-theoretic optimality among deterministic first-order methods was established by~\cite{DR17}. Prior works have not been able to analytically resolve the optimal solution to the second problem~\eqref{eq:subopttograd-design}, but the OGM-G algorithm of~\cite{KF21} is conjectured to attain it. Very recently, a conjectured solution to the third problem, called the Lemniscate method, was proposed and analyzed by~\cite{KimRyuDasGupta26}.

In this work, we take an algebraic (not information-theoretic) approach to understanding these problems. From this, we arrive at an algebraic proof that OGM solves~\eqref{eq:disttosubopt-design}, establishing for the first time that it does so uniquely. Second, applying ideas from the H-duality theory of~\cite{KOPR23}, we establish the optimality of OGM-G for~\eqref{eq:subopttograd-design} and prove its uniqueness. Finally, extending these techniques to the distance-to-gradient setting, we establish the optimality of the Lemniscate method for~\eqref{eq:disttograd-design} and again establish uniqueness.

Formally, we show
\begin{theorem}\label{thm:disttosubopt}
    If $d\geq N+2$, the OGM algorithm of~\cite{KF16} provides the unique matrix $W$ solving~\eqref{eq:disttosubopt-design}.\\
    In particular, the minimax optimal rate is $r_{D\rightarrow F} = 1/\theta_N^2 = \Theta(1/N^2)$, where $\theta_N$ is defined in~\eqref{eq:ogm-rate-recursion}.
\end{theorem}
\begin{theorem}\label{thm:subopttograd}
    If $d\geq N+2$, the OGM-G algorithm of~\cite{KF21} provides the unique matrix $W$ solving~\eqref{eq:subopttograd-design}.\\
    In particular, the minimax optimal rate is $r_{F\rightarrow G} = 1/\theta_N^2 = \Theta(1/N^2)$, where $\theta_N$ is defined in~\eqref{eq:ogm-rate-recursion}.
\end{theorem}
\begin{theorem}\label{thm:disttograd}
    If $d\geq N+2$, the Lemniscate algorithm of~\cite{KimRyuDasGupta26} provides the unique matrix $W$ solving~\eqref{eq:disttograd-design}.\\
    In particular, the minimax optimal rate is $r_{D\rightarrow G} = 1/\Omega_N^2 = \Theta(1/N^4)$, where $\Omega_N$ is defined in~\eqref{eq:lemni-rate-recursion}.
\end{theorem}
\noindent Note that the high-dimensionality assumption $d\geq N+2$ is standard in the dimension-independent analysis of first-order methods~\cite{nemirovsky1983problem}. This condition is also standard for the tight application of the Performance Estimation Problems (PEPs) due to~\cite{DT14,THG17,taylor2017composite}, upon which our work is based.

All three proofs follow the same algebraic template. Lemmas~\ref{lem:ogm-TC}, \ref{lem:ogmg-TC}, and~\ref{lem:lemni-TC}  factor the relevant PEP multiplier matrix into triangular factors. The subsequent elimination and projection lemmas reduce each design problem to an optimization problem in scalar variables with a unique optimizer. Its equality conditions determine the optimal triangular factors, after which a unique positive semidefinite completion determines the method matrix $W$. Appendix~\ref{app:rediscovery} shows that these equality conditions recover OGM, OGM-G, and Lemniscate acceleration without assuming their formulas in advance.

This approach contrasts with the traditional information-theoretic approach pioneered by~\cite{nemirovsky1983problem}, which underlies most lower bounds in first-order optimization theory. Our approach may provide new means to address other open complexity questions in optimization. Hence, it may be of independent interest beyond the particular theorems herein. An important limitation of this algebraic approach is that it only establishes optimality among fixed-step first-order methods (not larger classes of gradient-span or deterministic gradient methods addressed by classic zero-chain and resisting oracle approaches). It remains open, for example, whether OGM-G and Lemniscate are optimal among these larger classes. 

\paragraph{Outline.} Section~\ref{sec:prelim} provides preliminary results from performance estimation and linear algebra needed for our development. Proofs of results related to these preliminaries are given in our appendix for completeness. The following Sections~\ref{sec:ogm}--\ref{sec:lemni} then proceed to prove each of Theorems~\ref{thm:disttosubopt}--\ref{thm:disttograd}. Each of these proofs is structurally similar but diverges sufficiently in technical details that we present each independently.

\section{Preliminaries} \label{sec:prelim}

\paragraph{Performance Estimation Problems (PEPs).}
The PEP framework of~\cite{DT14,THG17,taylor2017composite} formulates the worst-case performance of a fixed algorithm, namely, the inner optimization problems in~\eqref{eq:disttosubopt-design}--\eqref{eq:disttograd-design}. For ease of introducing this approach, we present the PEP formulation for the inner problem of finding an instance with the worst-case final suboptimality in~\eqref{eq:disttosubopt-design}. Note this framework is applicable to many settings beyond smooth convex optimization, for which we refer readers to the collection of examples documented in PEPit~\cite{pepit2022}.

A key insight underlying PEP is that to understand a fixed-step first-order method's performance, one must only consider the function values $f_n=f(x_n)$ and gradient values $g_n=\nabla f(x_n)$ at the iterates $x_n$. Together, these values define a trace of the algorithm's trajectory $((x_i,f_i,g_i))_{i=0}^N$. Often we include the minimizer in this trace, which has $f_\star=f(x_\star)$ and $g_\star=0$. Hence, with observed-index set $\mathcal{I}=\{0,\dots,N\}$ and full index set $\mathcal{I}_\star=\{\star\}\cup\mathcal{I}$, a method's full trace is $((x_i,f_i,g_i))_{i\in\mathcal{I}_\star}$. 

The fundamental interpolation theorem for $L$-smooth convex functions~\cite{THG17} states that a trace can be interpolated by some $L$-smooth convex function if and only if the following quantities are nonnegative
\begin{equation*}
    \mathcal{Q}_{i,j} = f_i - f_j - \langle g_j, x_i-x_j\rangle - \frac{1}{2L}\|g_i-g_j\|^2 \qquad \forall i\neq j\in\mathcal{I}_\star.
\end{equation*}

When $x_i$ and $x_j$ are generated by a fixed algorithm $W$, we can view $\mathcal{Q}_{i,j}$ as an affine function: Eliminating each $x_n$ for $n=1,\dots,N$ via~\eqref{eq:fixed-step-definition}, observe that $\mathcal{Q}_{i,j}(F,G; W)$ is an affine function of the vector $F=(f_0-f_\star,\dots,f_N-f_\star)$ and the positive semidefinite matrix $G = P^\top P$ where $P=[x_0-x_\star,g_0,\dots,g_N]$. For example, consider the suboptimality minimization task of~\eqref{eq:disttosubopt-design}. Noting that for $d\geq N+2$, any positive semidefinite $G$ can be factored to recover $P$, the inner maximization of~\eqref{eq:disttosubopt-design} can be equivalently expressed as
\begin{equation} \label{eq:PEP-SDP}
\begin{aligned}
    \sup_{F,G}\quad & f_N-f_\star\\
    \mathrm{s.t.}\quad
    &\mathcal{Q}_{i,j}(F,G;W)\geq0
    \qquad \forall i\neq j\in\mathcal{I}_\star,\\
    &\mathcal{D}(G)\geq0,\\
    &G\succeq0
\end{aligned}
\end{equation}
where $\mathcal{D}(G) = 1 - \frac{1}{2}\|x_0-x_\star\|^2$ is affine in $G$. For fixed $W$, this is a semidefinite program. In later sections when considering minimization of the squared gradient norm, similar semidefinite programs will be provided specialized to their goals.

Dually, this can be approached by considering nonnegative off-diagonal multipliers $\lambda_{i,j}$, a nonnegative multiplier $\sigma$, and a positive semidefinite matrix $Z$. To prove an upper bound $r$, these multipliers must satisfy the identity
\begin{equation} \label{eq:main-identity}
    r - (f_N-f_\star) = \sum_{i\neq j\in\mathcal{I}_\star} \lambda_{i,j}\mathcal{Q}_{i,j}(F,G;W) + \sigma \mathcal{D}(G) + \langle Z,G\rangle
\end{equation}
for all vectors $F$ and symmetric matrices $G$. 
Since both sides above are affine in $F,G$, enforcing that this identity holds is a linear system in $\lambda,\sigma,Z$ with one equation per coefficient. Later, when we consider joint optimization of $W$ as well, this identity will become bilinear.

Note that $(\lambda_{i,j})_{i,j\in\mathcal{I}_\star}$ include all ordered pairs involving $\star$. Its off-diagonal entries are the multipliers of the nontrivial interpolation constraints. Since $\mathcal{Q}_{i,i}$ is identically zero, we can set $\lambda_{i,i}$ freely (even as a negative value). For the suboptimality PEP, we enforce the convention that
\begin{equation}\label{eq:distance-lambda-diagonal}
    \lambda_{i,i}=-\sum_{j\in\mathcal{I}_\star\setminus\{i\}}\lambda_{j,i}
    \qquad \forall i\in\mathcal{I}.
\end{equation}
This diagonal convention and the previous convention for $W$ to have unit diagonal entries, while arbitrary, will simplify the subsequent analysis.

Then the dual semidefinite program to~\eqref{eq:PEP-SDP} for any fixed $W$ is
\begin{equation} \label{eq:PEP-SDP-dual-with-identity}
\begin{aligned}
    \inf_{\lambda,\sigma,r,Z}\quad & r\\
    \mathrm{s.t.}\quad
    &\text{identity~\eqref{eq:main-identity} holds},\\
    &\lambda_{i,j}\geq0 \quad \forall i\neq j\in\mathcal{I}_\star,\\
    &\sigma\geq0, \qquad Z\succeq0.
\end{aligned}
\end{equation}

The identity constraint in~\eqref{eq:PEP-SDP-dual-with-identity} can be solved explicitly. To this end, let $\Lambda\in\mathbb{R}^{(N+1)\times(N+1)}$ denote the matrix $(\lambda_{i,j})_{i,j\in\mathcal I}$, omitting multipliers involving $\star$, and let $\operatorname{sym}(A)=(A+A^\top)/2$. The resulting simplified dual statement is given below; its proof is deferred to Appendix~\ref{app:distance-dual}.

\begin{proposition}\label{prop:distance-dual}
For any unit lower triangular matrix $W$, fixing $L=1$, the PEP~\eqref{eq:PEP-SDP} has dual
\begin{equation} \label{eq:PEP-SDP-dual-clean}
\begin{aligned}
    \inf_{\Lambda,r}\quad & r\\
    \mathrm{s.t.}\quad
    &\Lambda_{i,j}\geq0
    \qquad \forall i\neq j\in\mathcal I,\\
    &\Lambda^\top\mathbf1\leq0,
    \qquad
    \Lambda\mathbf1\leq-e_N,\\
    &\begin{pmatrix}
        \frac r2 & \frac12\mathbf1^\top\Lambda\\[1mm]
        \frac12\Lambda^\top\mathbf1
        &-\operatorname{sym}(\Lambda^\top W)-\frac12e_Ne_N^\top
    \end{pmatrix}\succeq0.
\end{aligned}
\end{equation}
\end{proposition}
\noindent Strong duality holds between this primal--dual pair. The following proposition states this with proof deferred to Appendix~\ref{app:distance-duality}.
\begin{proposition}\label{prop:distance-duality}
For any method $W$, the primal PEP~\eqref{eq:PEP-SDP} and the dual PEP~\eqref{eq:PEP-SDP-dual-clean} have equal value.
\end{proposition}

\noindent We can therefore plug this dual problem into
\eqref{eq:disttosubopt-design} to reformulate the algorithm design problem as
\begin{equation}\label{eq:disttosubopt-form2}
\begin{aligned}
    r_{D\rightarrow F}
    =\inf_{W,\Lambda,r}\quad & r\\
    \mathrm{s.t.}\quad
    &W\text{ is unit lower triangular},\\
    &\Lambda_{i,j}\geq0
    \qquad \forall i\neq j\in\mathcal I,\\
    &\Lambda^\top\mathbf1\leq0,
    \qquad
    \Lambda\mathbf1\leq-e_N,\\
    &\begin{pmatrix}
        \frac r2 & \frac12\mathbf1^\top\Lambda\\[1mm]
        \frac12\Lambda^\top\mathbf1
        &-\operatorname{sym}(\Lambda^\top W)-\frac12e_Ne_N^\top
    \end{pmatrix}\succeq0,
    \qquad r>0.
\end{aligned}
\end{equation}

\paragraph{$M$-Matrices and Triangular Factorizations.}
We denote the all-ones vector by $\mathbf{1}$ and the standard basis vectors by $e_0,\dots,e_N$. A real square matrix $A$ is a nonsingular $M$-matrix if its off-diagonal entries are nonpositive, it is nonsingular, and $A^{-1}\geq0$ entrywise. Equivalently, all principal minors of $A$ are positive; see~\cite[Chapter~6, Theorem~2.3]{BermanPlemmons94}. The (negated) PEP multiplier blocks $-\Lambda$ we consider often possess this structure but may lie on this set's singular boundary. The following proposition gives a useful triangular factorization covering these potentially singular matrices. The subsequent lemma supplies inverse-entry comparisons, again without requiring the whole matrix to be nonsingular. Proofs of these two results follow from standard linear algebra, but are given in Appendix~\ref{app:LA} for completeness.

\begin{proposition}
\label{prop:M-matrix-factorization}
Let $A$ have nonpositive off-diagonal entries, with
$A\mathbf1\geq0$ and $\mathbf1^\top A\geq0$. Then $A$ admits a factorization
\[
A=TC,
\]
where $T$ is upper triangular and $C$ is unit lower triangular, with
\[
T_{i,i}\geq0,\qquad
T_{i,j}\leq0 \quad \forall i < j,\qquad
C_{i,j}\leq0 \quad \forall i > j.
\]
Moreover, $C^{-1}\geq0$ and $\mathbf1^\top T\geq0$ entrywise. If $A$ is nonsingular, then the factorization is unique with
\[
T_{i,i}>0,
\qquad
T^{-1}\geq0,
\qquad
C^{-1}\geq0.
\]
If $A\mathbf1=e_0$, the factorization may be chosen so that $C\mathbf1=e_0$ and $T_{0,0}=1$. If, instead, $\mathbf1^\top A=e_0^\top$, then they may be chosen such that $\mathbf1^\top T=e_0^\top$ and $T_{0,0}=1$.
\end{proposition}

\begin{lemma}\label{lem:app-inverse-bound}
Let $A$ have nonpositive off-diagonal entries, with
$A\mathbf1\geq0$ and $\mathbf1^\top A\geq0$. Whenever
\[
Au=e_i,
\qquad
Av=e_j
\]
are solvable, every solution has $0\leq u_j\leq\min\{u_i,v_j\}$.
The same statement holds for every principal submatrix of $A$. In addition, $\operatorname{range}(A)=\operatorname{range}(A^\top)$.
\end{lemma}

\subsection{Related Work}

\paragraph{Optimized gradient methods.}
OGM originated from the PEP computations
of \cite{DT14}; its efficient analytic form and convergence proof were
developed by \cite{KF16}.
Subsequent work clarified the method's mechanism. Analyses using Lyapunov functions and
linear coupling \cite{park2021factor,d2021acceleration}
explain the factor improvement of OGM over the classical
accelerated gradient method without relying on a computer-assisted
certificate.

For the gradient-norm criterion, Nesterov~\cite{nesterov2012make} initiated the study of accelerated convergence guarantees in terms of gradient norms. Eventually, Kim and Fessler~\cite{KF21} designed OGM-G, the gradient-norm (relative to initial suboptimality) counterpart of OGM, by optimizing PEP certificates for the squared gradient norm. The geometric framework of~\cite{lee2021geometric} gave a Lyapunov-style view of OGM-G. It also introduced a variant of the fast iterative shrinkage-thresholding algorithm (FISTA)~\cite{beck2009fast}, called FISTA-G, for reducing the squared norm of the proximal gradient mapping, along with related methods. 
For the squared gradient norm (relative to initial squared distance to a minimizer), it was observed in~\cite{KF21} that running OGM followed by OGM-G provided a fast $O(1/N^4)$ rate, but did not appear to be exactly numerically optimal. Recently,~\cite{KimRyuDasGupta26} designed and analyzed Lemniscate acceleration as an optimized method for this setting. Before this work, exact minimax optimality had been established for OGM, while the corresponding claims for OGM-G and Lemniscate remained conjectural. Our work resolves these missing lower bounds and establishes uniqueness among fixed-step first-order methods in each case.

We remark that optimal methods have been designed for many settings beyond the smooth convex problems considered here. For example, the information-theoretic exact method (ITEM)~\cite{taylordrori2022} and ITEM-f method~\cite{KimRyuDasGupta26} are tailored to smooth, strongly convex settings, the OptISTA method~\cite{jang2025computer} applies to composite settings, and the Prox-ITEM method~\cite{upadhyaya2026optimalfirstordermethodsmooth} applies to strongly convex composite settings. Whether similar results to those developed here can provide uniqueness guarantees (or demonstrate multiple solutions exist) for these wider settings is left as an interesting future direction.

A general constructive technique for analytically identifying such optimal methods was provided by Drori and Taylor~\cite{DT20}. A general numerical approach was developed by~\cite{DGVR24} via branch-and-bound solutions to pure minimization formulations like~\eqref{eq:disttosubopt-form2}.
Our proofs share some core structures with these prior general approaches. We begin with the same minimization forms like~\eqref{eq:disttosubopt-form2}, but then analytically reduce them to a form simple enough to be solved exactly, from which uniqueness follows as well. This strategy applies equally to resolve the optimality and uniqueness of OGM, OGM-G, and Lemniscate.

\paragraph{H-duality and related dualities.}
\cite{KOPR23} introduced H-duality as a correspondence between methods
for reducing function value and methods for reducing gradient norm.
In that theory, OGM and OGM-G form the central example: the H-dual
operation transports a Lyapunov proof for one performance criterion to a
certificate for the other. Shu and Wang~\cite{Shu2026Hduality} provided general explanations for this phenomenon and a mapping between generic PEP proofs under this operation (subject to the verification that sign conditions are preserved).

The mirror-duality framework of
\cite{KPDOR24} extends the function-value/gradient-norm symmetry to
methods based on mirror descent and Banach-space geometries. H-duality has
also appeared beyond smooth convex minimization: \cite{YKSR24} use it
to exhibit nonunique optimal acceleration mechanisms for minimax and
fixed-point problems, and \cite{YG26} study the H-dual operation within
a combinatorial theory of extremal optimal fixed-point algorithms.
Here a similar mirrored phenomenon occurs. Our analysis of the algorithm design problems for both suboptimality and squared gradient norm reduces to the same core problem, up to a mirroring in index ordering.

\paragraph{Neighboring exact-optimality questions.}
Recent works characterizing exactly optimal methods have emphasized that minimax optimality and uniqueness are separate questions. For the design of methods for nonexpansive fixed-point problems, the H-invariance theory of
\cite{YRG25} and the fixed-point duality theory of \cite{YG26}
identified the exact set of optimal methods and structured relationships among them. For nonsmooth Lipschitz
convex minimization, \cite{ZG26} similarly gave a complete characterization of
all minimax optimal fixed-step subgradient methods. One can view our results as providing a complete characterization of the set of optimal methods for three settings of smooth convex optimization. In these cases, the sets of optimal methods are singletons.

Beyond our pursuit of optimal fixed-step first-order methods, algorithm design can target stronger criteria. The work~\cite{GSW26} constructed a Subgame Perfect Gradient Method (SPGM) that matches OGM's
static minimax rate but optimally improves its guarantees after
informative oracle observations. The possible uniqueness of such adaptively optimal policies is interesting but beyond our scope here. 

\section{Uniqueness and Optimality of OGM}\label{sec:ogm}

Recall our notation that $e_i$ denotes the $i$th coordinate vector,
$\mathbf 1$ denotes the all-ones vector, and
$\operatorname{sym}(A)=(A+A^\top)/2$. All matrices are indexed by
$0,\dots,N$.

For a given budget of steps $N$, the OGM method is parameterized by a scalar sequence $\theta_n$ defined by first setting $\theta_0=1$. Then for $n=1,\dots,N-1$ and $n=N$, set $\theta_n$ as the positive root of the following quadratic equations, respectively,
\begin{equation}\label{eq:ogm-rate-recursion}
\theta_n^2-\theta_n-\theta_{n-1}^2=0,\qquad 
\theta_N^2-\theta_N-2\theta_{N-1}^2=0.
\end{equation}
The OGM algorithm can then be defined by the unit lower triangular fixed-step matrix $W_\mathtt{OGM}$
\begin{equation}\label{eq:ogm-W}
(W_\mathtt{OGM})_{n,n}=1,
\qquad
(W_\mathtt{OGM})_{n,i}=
\begin{cases}
\frac{\theta_i^2}{\theta_n^2}
+2\theta_i
\left(1-\frac{\theta_i^2}{\theta_n^2}\right),
&0\leq i<n<N,\\
\frac{2\theta_i^2}{\theta_N^2}
+2\theta_i
\left(1-\frac{2\theta_i^2}{\theta_N^2}\right),
&0\leq i<n=N.
\end{cases}
\end{equation}
In addition to this explicit matrix form, one can equivalently describe OGM recursively\footnote{The OGM algorithm can equally be defined as updating its scalar sequence by the recurrence, $n=1,\dots,N-1$,
\[
\theta_n=\frac{1+\sqrt{1+4\theta_{n-1}^2}}{2},
\qquad
\theta_N=\frac{1+\sqrt{1+8\theta_{N-1}^2}}{2}
\]
and maintaining an auxiliary iterate sequence, initialized $z_0=x_0$ and updating for $n=1,\dots,N$
\begin{equation*}
\begin{aligned}
z_{n}=z_{n-1}-\frac{2\theta_{n-1}}{L}g_{n-1}, \qquad
x_{n}
=\begin{cases}
\frac{\theta_{n-1}^2}{\theta_n^2}\left(x_{n-1} - \frac{1}{L}g_{n-1}\right)
+\left(1-\frac{\theta_{n-1}^2}{\theta_n^2}\right)z_{n}
&\text{if }n<N\\
\frac{2\theta_{N-1}^2}{\theta_N^2}\left(x_{N-1} - \frac{1}{L}g_{N-1}\right)
+\left(1-\frac{2\theta_{N-1}^2}{\theta_N^2}\right)z_N
& \text{if }n=N.
\end{cases}
\end{aligned}
\end{equation*}
}.
Kim and Fessler~\cite{KF16} showed that the OGM algorithm possesses a final suboptimality guarantee of $f_N-f_\star\leq 1/\theta_N^2$. We denote this optimal rate by $r_\mathtt{OGM}=1/\theta_N^2$.

We note that this sequence is also related to the following simple optimization problem over the simplex $\Delta_{N+1} :=\{\delta\in\mathbb{R}^{N+1}_+:\mathbf 1^\top\delta=1\}$ with the
convention $0/0=0$. In particular, the sequence $(\theta_i)_{i=0}^N$ determines its unique solution. This fact will play a central role in the final steps of our analysis of both OGM and OGM-G.

\begin{lemma}[The Simplex Problem]\label{lem:ogm-simplex-prob}
For every $N\geq1$, the optimization problem
\begin{equation*}
\begin{aligned}
\min_{\delta\in\Delta_{N+1}}\quad
&\Phi(\delta) := \max\left\{
\frac{\delta_i^2}{2\sum_{j=0}^i\delta_j}:0\leq i<N,
\ \delta_N^2
\right\}
\end{aligned}
\end{equation*}
has optimal value $r_\mathtt{OGM}$, attained uniquely at $\delta^\star$ defined for $i=0,\dots,N-1$ and $i=N$ as
\begin{equation}\label{eq:ogm-chain-increments}
    \delta_i^\star=2r_\mathtt{OGM}\theta_i,
    \qquad
    \delta_N^\star=\sqrt{r_\mathtt{OGM}}.
\end{equation}
Moreover, at $\delta^\star$, every term in the finite maximum defining $\Phi(\delta^\star)$ equals $r_\mathtt{OGM}$.
\end{lemma}
\begin{proof}
Consider some $\delta\in \Delta_{N+1}$. Then by simple rearrangement of the definition of $\Phi(\delta)$, for every
$i<N$, one has that $\delta_i \leq\sqrt{2 \Phi(\delta) \sum_{j=0}^i\delta_j}$.
We claim inductively that
\begin{equation}\label{eq:ogm-chain-envelope}
    \sum_{j=0}^i\delta_j\leq2\Phi(\delta)\theta_i^2
    \qquad \forall 0\leq i<N.
\end{equation}
For $i=0$, the constraint $\delta_0^2/(2\delta_0)\leq\Phi(\delta)$ gives
$\delta_0\leq2\Phi(\delta)$. Suppose the claim holds at $i-1$. Since
$\delta_i=\sum_{j=0}^i\delta_j-\sum_{j=0}^{i-1}\delta_j$, we have
\[
    \sum_{j=0}^i\delta_j
    -\sqrt{2\Phi(\delta)\sum_{j=0}^i\delta_j}
    \leq2\Phi(\delta)\theta_{i-1}^2.
\]
Writing
$z=\sqrt{(\sum_{j=0}^i\delta_j)/(2\Phi(\delta))}$, this becomes
$z^2-z\leq\theta_{i-1}^2$. The first quadratic in~\eqref{eq:ogm-rate-recursion} shows that its positive root is $\theta_i$. This proves
\eqref{eq:ogm-chain-envelope}. Similarly, the terminal constraint gives $\delta_N\leq\sqrt{\Phi(\delta)}$. Hence
\[
    1
    =\sum_{j=0}^{N-1}\delta_j+\delta_N
    \leq2\Phi(\delta)\theta_{N-1}^2+\sqrt{\Phi(\delta)}.
\]
The right side is strictly increasing in $\Phi(\delta)$, and the terminal equation
in~\eqref{eq:ogm-rate-recursion} together with $r_\mathtt{OGM}=1/\theta_N^2$ gives
$1=2r_\mathtt{OGM}\theta_{N-1}^2+\sqrt{r_\mathtt{OGM}}$. Therefore $\Phi(\delta)\geq r_\mathtt{OGM}$.
For the point in~\eqref{eq:ogm-chain-increments}, the ordinary recurrence gives $\sum_{j=0}^i\delta_j^\star=2r_\mathtt{OGM}\theta_i^2$ for each $i<N$ and the terminal recurrence gives
$\sum_{j=0}^N\delta_j^\star=1$. Consequently, for $i=0,\dots,N-1$ and $i=N$, we have
\[
    \frac{(\delta_i^\star)^2}
    {2\sum_{j=0}^i\delta_j^\star}
    =r_\mathtt{OGM},
    \qquad
    (\delta_N^\star)^2=r_\mathtt{OGM}.
\]
Thus the lower bound is attained.
Finally, suppose a feasible $\delta$ attains $\Phi(\delta)=r_\mathtt{OGM}$. Equality in the
terminal bound forces
$\delta_N=\sqrt{r_\mathtt{OGM}}$ and
$\sum_{j=0}^{N-1}\delta_j=2r_\mathtt{OGM}\theta_{N-1}^2$. If
$\sum_{j=0}^i\delta_j=2r_\mathtt{OGM}\theta_i^2$, then
\[
    \sum_{j=0}^{i-1}\delta_j
    \geq
    2r_\mathtt{OGM}\theta_i^2-2r_\mathtt{OGM}\theta_i
    =2r_\mathtt{OGM}\theta_{i-1}^2,
\]
while~\eqref{eq:ogm-chain-envelope} gives the reverse inequality.
Backward induction therefore fixes every partial sum and hence every
coordinate of $\delta$, proving uniqueness.
\end{proof}

\subsection{A Direct Triangular Reformulation}

As a key conceptual step in our analysis, we rewrite the negated matrix $-\Lambda$ as a product of two triangular matrices, $TC$. Note that the factorization in Proposition~\ref{prop:M-matrix-factorization} applies for any feasible matrix, even when $\Lambda$ is singular. Doing so gives the following reformulation.

\begin{lemma}[The $TC$ reformulation]\label{lem:ogm-TC}
The value in~\eqref{eq:disttosubopt-design} is
\begin{equation}\label{eq:ogm-TC-problem}
\begin{aligned}
 r_{D\rightarrow F}
 =\inf_{W,T,C,r}\quad &r\\
 \mathrm{s.t.}\quad
 &W\text{ is unit lower triangular},\\
 &T\text{ is upper triangular},\qquad T_{i,i}\geq0,\qquad T_{i,j}\leq0 \quad\forall i<j,\\
 &C\text{ is unit lower triangular},\qquad C_{i,j}\leq0\quad \forall i>j,\\
 &(TC)_{i,j}\leq0\quad \forall i\neq j,\qquad TC\mathbf 1\geq e_N,
 \qquad C^\top T^\top\mathbf 1\geq0,\\
 &\begin{pmatrix}
       \frac r2 &-\frac12\mathbf 1^\top TC\\[1mm]
       -\frac12C^\top T^\top\mathbf 1
       &\operatorname{sym}(W^\top TC)-\frac12e_Ne_N^\top
   \end{pmatrix}\succeq0, \qquad r>0.
\end{aligned}
\end{equation}
For each fixed $W$, factoring $-\Lambda=TC$ identifies each feasible dual PEP solution with one for the above reformulation.
\end{lemma}

\begin{proof}
Let $(W,\Lambda,r)$ be feasible in~\eqref{eq:disttosubopt-form2}. Then
$A=-\Lambda$ has nonpositive off-diagonal entries, with
\[
    A\mathbf1\geq e_N,
    \qquad
    \mathbf1^\top A\geq0.
\]
Proposition~\ref{prop:M-matrix-factorization} gives a factorization $A=TC$
with the stated triangular forms and signs. Substituting $\Lambda=-TC$ into
the scalar and semidefinite constraints of~\eqref{eq:disttosubopt-form2}
gives exactly the remaining constraints in~\eqref{eq:ogm-TC-problem}.
Conversely, given feasible $(W,T,C,r)$, setting $\Lambda=-TC$ gives a
feasible tuple in~\eqref{eq:disttosubopt-form2}. 
\end{proof}

For fixed $T,C$, this reformulation allows us to view the semidefinite constraint as a triangular positive semidefinite (PSD) completion. From this, we can eliminate $W$.

\begin{lemma}[Eliminating $W$]\label{lem:ogm-eliminate-W}
Fix admissible $T,C$ and $r>0$. There exists a unit lower triangular $W$
satisfying the semidefinite constraint in~\eqref{eq:ogm-TC-problem} if and
only if
\begin{equation}\label{eq:ogm-diagonal-conditions}
\frac{(\mathbf 1^\top Te_i)^2}{2r}
\leq
T_{i,i}-\frac12(e_N^\top C^{-1}e_i)^2
\qquad \forall 0\leq i\leq N.
\end{equation}
If $T$ is nonsingular and every inequality is tight, then this $W$ is unique.
\end{lemma}

\begin{proof}
Taking the Schur complement of the semidefinite constraint and congruence by $C^{-1}$ gives
\[
 \operatorname{sym}(T^\top WC^{-1})
 -\frac12qq^\top-\frac1{2r}pp^\top\succeq0,
 \qquad q=C^{-\top}e_N,\quad p=T^\top\mathbf1.
\]
Its diagonal being nonnegative gives~\eqref{eq:ogm-diagonal-conditions}.
Conversely, suppose~\eqref{eq:ogm-diagonal-conditions} holds and let $L$ be
lower triangular with
\[
 L_{i,i}=T_{i,i},
 \qquad
 L_{i,j}=q_iq_j+\frac1r p_ip_j
 \quad \forall i>j.
\]
Then $\operatorname{sym}(L)-qq^\top/2-pp^\top/(2r)$ is diagonal and positive
semidefinite. If $T_{i,i}=0$, the corresponding condition gives $p_i=q_i=0$,
and the sign constraints force the $i$th column of $T$ to be zero. Hence one
can solve $T^\top X=L$ successively by rows with $X$ unit lower triangular,
taking the $i$th row of $X$ to be $e_i^\top$ whenever $T_{i,i}=0$. Setting
$W=XC$ proves
sufficiency.

If $T$ is nonsingular and every inequality is tight, then the slack positive semidefinite matrix $\operatorname{sym}(T^\top WC^{-1})
-\frac12qq^\top-\frac1{2r}pp^\top$ must have all zero diagonal. Hence it must be the zero matrix. Since $T^\top WC^{-1}$ is lower triangular, its symmetrization uniquely determines it. Then the bijection $W\mapsto T^\top WC^{-1}$ uniquely determines $W$.
\end{proof}

\subsection{Projection onto the Simplex Problem}

Ultimately, we will solve~\eqref{eq:ogm-TC-problem}, finding the unique
optimal factors $T_F^\star,C_F^\star$ and optimal value
$r_\mathtt{OGM}$. Namely, let $C_F^\star=I$, and let $T_F^\star$ be the upper
triangular matrix with the only nonzero entries given by
\begin{equation}\label{eq:ogm-T-star-r}
    (T_F^\star)_{i,i}=
    \begin{cases}
        2r_\mathtt{OGM}\theta_i^2,&0\leq i<N,\\
        1,&i=N,
    \end{cases}
    \qquad
    (T_F^\star)_{i,i+1}=-(T_F^\star)_{i,i}
    \qquad \forall 0\leq i<N.
\end{equation}
The following lemma establishes that $r_\mathtt{OGM}$ is a lower bound
on~\eqref{eq:ogm-TC-problem} and that, if it is attained, the associated
factors must equal $T_F^\star,C_F^\star$. It also shows that every condition
in~\eqref{eq:ogm-diagonal-conditions} is then tight, so
Lemma~\ref{lem:ogm-eliminate-W} will uniquely determine $W$. Doing so, one can (re)discover OGM, given in Appendix~\ref{subsec:discovering-ogm}.

\begin{lemma}[Projection onto Simplex Problem]\label{lem:ogm-TC-projection}
Let $W,T,C,r$ be feasible in~\eqref{eq:ogm-TC-problem}. Then there is a
vector $b$ satisfying
\begin{equation}\label{eq:ogm-backward-solution}
    Tb=e_N,
    \qquad
    0\leq b\leq\mathbf1.
\end{equation}
Define $\delta_i=(\mathbf 1^\top Te_i)b_i$ for each $i=0,\dots,N$.
Then $\delta\in\Delta_{N+1}$ and
\begin{equation}\label{eq:ogm-projected-chain}
    \max\left\{
        \frac{\delta_i^2}{2\sum_{j=0}^i\delta_j}:0\leq i<N,
        \ \delta_N^2
    \right\}
    \leq r.
\end{equation}
Consequently, $r\geq r_\mathtt{OGM}$. If $r=r_\mathtt{OGM}$, then
$T=T_F^\star$, $C=C_F^\star$, and every inequality
in~\eqref{eq:ogm-diagonal-conditions} is tight.
\end{lemma}
\begin{proof}
First, we construct the desired vector $b$, solving the system of equations $Tb=e_N$.
Set $p=T^\top\mathbf1$ and $y=C\mathbf1$.
Since $C^{-1}\geq0$ and
$C^\top p=C^\top T^\top\mathbf1\geq0$, we have $p\geq0$.
Also $Ty\geq e_N$, while $y\leq\mathbf1$.
From this, it follows that if $T_{i,i}=0$, then
\[
    0\leq p_i=T_{i,i}+\sum_{j<i}T_{j,i}\leq0,
\]
and so the entire $i$th column of $T$ is zero. Set $\bar y_i=y_i$ when
$T_{i,i}>0$ and $\bar y_i=0$ otherwise. Then
$T\bar y=Ty\geq e_N$ and backward induction gives $\bar y\geq0$.

Now we proceed to construct $b$ via backward substitution on the system $Tb=e_N$, considering the equations $i=N,N-1,\dots,0$. When $T_{i,i}>0$, we can directly set $b_i$ to complete the $i$th equation with $0\leq b_i\leq\bar y_i$.
When $T_{i,i}=0$, the $i$th equation does not involve $b_i$. We set $b_i=0$. To verify our construction solves the $i$th equation in this case, let $\tau_i = \sum_{j>i} T_{i,j}b_j$, which we seek to show equals zero. Since $T_{i,j}\leq 0$ and $b_j\geq 0$, we have $\tau_i\leq 0$. However, since $b_j \leq \bar{y}_j$, we have $\tau_i \geq \sum_{j>i} T_{i,j}\bar{y}_j = (T\bar{y})_i \geq 0$. Hence $\tau_i=0$ and the $i$th equation is satisfied. Thus, we have constructed $b$ satisfying~\eqref{eq:ogm-backward-solution}.

It follows that $\delta\geq0$ and
$\sum_i\delta_i=\mathbf1^\top Tb=1$. For $i<N$, expanding $Tb=e_N$
gives
\begin{equation}\label{eq:ogm-cut-r}
    \sum_{j=0}^i\delta_j
    =T_{i,i}b_i
    +\sum_{k=0}^{i-1}\sum_{j=i+1}^N(-T_{k,j})b_j
    \geq T_{i,i}b_i.
\end{equation}
Together with~\eqref{eq:ogm-diagonal-conditions} and $b_i\leq1$, this
gives
\[
    \frac{\delta_i^2}{2\sum_{j=0}^i\delta_j}
    \leq
    \frac{(\mathbf1^\top Te_i)^2b_i}{2T_{i,i}}
    \leq
    \frac{(\mathbf1^\top Te_i)^2}{2T_{i,i}}
    \leq r
    \qquad \forall 0\leq i<N,
\]
under the convention $0/0=0$. At $i=N$, the last row gives
$b_N=1/T_{N,N}$. Since $b_N\leq y_N\leq1$, we have $T_{N,N}\geq1$.
Since $e_N^\top C^{-1}e_N=1$, the terminal diagonal condition gives
\[
    \delta_N^2
    \leq r\frac{2T_{N,N}-1}{T_{N,N}^2}
    \leq r.
\]
This proves~\eqref{eq:ogm-projected-chain}, and
Lemma~\ref{lem:ogm-simplex-prob} gives
$r\geq r_\mathtt{OGM}$.

Suppose now that $r=r_\mathtt{OGM}$.
Lemma~\ref{lem:ogm-simplex-prob} gives
$\delta=\delta^\star$, with every term in
\eqref{eq:ogm-projected-chain} tight. If $T_{i,i}=0$, then the $i$th
column of $T$ vanishes and $\delta_i=0$. Since every coordinate of
$\delta^\star$ is positive, $T_{i,i}>0$ for every $i$. For $i<N$,
equality in the chain gives
\[
    b_i=1,
    \qquad
    T_{i,i}=\sum_{j=0}^i\delta_j^\star
    =2r_\mathtt{OGM}\theta_i^2,
    \qquad
    e_N^\top C^{-1}e_i=0.
\]
The terminal equality gives $T_{N,N}=b_N=1$. Hence $b=\mathbf1$.
Since $b\leq C\mathbf1\leq\mathbf1$, we obtain
$C\mathbf1=\mathbf1$, and the sign constraints force $C=I$.
Equality in~\eqref{eq:ogm-cut-r} forces every entry of $T$ above the first
superdiagonal to vanish, and $T\mathbf1=e_N$ gives
$T_{i,i+1}=-T_{i,i}$. Thus $T=T_F^\star$ and $C=C_F^\star$.
Reading the equality chain backward also shows that every inequality
in~\eqref{eq:ogm-diagonal-conditions} is tight.
\end{proof}

The optimal factors $T_F^\star$ and $C_F^\star$ are not only unique, if attained, but also are the limit of any sequence approaching attainment of the $r_\mathtt{OGM}$ lower bound. The following proposition formalizes this stability result. The proof is a routine extension of the above uniqueness argument, and so deferred to Appendix~\ref{app:limiting-factors} for completeness.
\begin{proposition}[Uniqueness of Limiting Optimal Factors]
\label{prop:ogm-limiting-factors}
Let $W$ be fixed. If $(W,T^{(k)},C^{(k)},r^{(k)})$ is feasible in
\eqref{eq:ogm-TC-problem} and $r^{(k)}\to r_\mathtt{OGM}$, then $T^{(k)}\to T_F^\star$ and $C^{(k)}\to C_F^\star$.
\end{proposition}

\subsection{Proof of Theorem~\ref{thm:disttosubopt}}
Fix a method $W$, and let $v_{D\rightarrow F}(W)$ denote the exact value of its PEP. By
Proposition~\ref{prop:distance-duality}, there is a sequence
$(\Lambda^{(k)},r^{(k)})$ feasible in
\eqref{eq:PEP-SDP-dual-clean} for this fixed $W$, with $r^{(k)}>0$ and
$r^{(k)}\to v_{D\rightarrow F}(W)$. Factor
\[
    -\Lambda^{(k)}=T^{(k)}C^{(k)}
\]
using Proposition~\ref{prop:M-matrix-factorization}.
Lemma~\ref{lem:ogm-TC} makes
$(W,T^{(k)},C^{(k)},r^{(k)})$ feasible in
\eqref{eq:ogm-TC-problem}, and
Lemma~\ref{lem:ogm-TC-projection} gives
$r^{(k)}\geq r_\mathtt{OGM}$. Passing to the limit shows $v_{D\rightarrow F}(W)\geq r_\mathtt{OGM}$. Suppose equality holds. Proposition~\ref{prop:ogm-limiting-factors} gives
\[
    T^{(k)}\to T_F^\star,
    \qquad
    C^{(k)}\to C_F^\star.
\]
Passing to the limit in the semidefinite constraint of
\eqref{eq:ogm-TC-problem} shows that
$(W,T_F^\star,C_F^\star,r_\mathtt{OGM})$ is feasible. The matrix
$T_F^\star$ is nonsingular, and every inequality in
\eqref{eq:ogm-diagonal-conditions} is tight. Hence
Lemma~\ref{lem:ogm-eliminate-W} shows that $W$ is the unique unit lower
triangular matrix satisfying the semidefinite constraint for these
factors.

The guarantee in~\cite{KF16} gives $v_{D\rightarrow F}(W_\mathtt{OGM})\leq r_\mathtt{OGM}$.
Together with the preceding lower bound, this gives
$v_{D\rightarrow F}(W_\mathtt{OGM})=r_\mathtt{OGM}$. Moreover, every optimal $W$ must equal $W_\mathtt{OGM}$.

\section{Uniqueness and Optimality of OGM-G}\label{sec:ogmg}

Next, we address the complementary setting of designing optimal algorithms for reducing the final squared gradient norm relative to initial suboptimality. Below we introduce the OGM-G algorithm. Then we present PEP in terms of dual multipliers for minimizing the squared gradient norm. 

Let $S\in\mathbb{R}^{(N+1)\times(N+1)}$ be the fixed cumulative reversal
matrix
\begin{equation*}
    S_{i,j}=
    \begin{cases}
        1,&i+j\geq N,\\
        0,&i+j<N.
    \end{cases}
\end{equation*}
The matrix $S$ is symmetric and invertible. For a unit lower triangular
method matrix $W$, define its H-dual method matrix by
\begin{equation*}
    \widehat W=SW^\top S^{-1}.
\end{equation*}
This is again unit lower triangular, and the map $W\mapsto\widehat W$ is an
involution.

In terms of the OGM matrix in~\eqref{eq:ogm-W}, OGM-G is defined by the matrix
\begin{equation*}
    W_\mathtt{OGM-G}
    =
    S W_\mathtt{OGM}^\top S^{-1}.
\end{equation*}
Equivalently, its H-dual method matrix is precisely the OGM matrix
$\widehat W_\mathtt{OGM-G}=W_\mathtt{OGM}$. This operation can be seen more
simply by defining the algorithm instead by an ``$H$ matrix'' such that
\begin{equation*}
    x_n=x_{n-1}-\frac{1}{L}\sum_{i=0}^{n-1}H_{n,i}g_i
\end{equation*}
where $H_{n,i}=W_{n,i}-W_{n-1,i}$ for $i=0,\dots,n-2$ and $H_{n,n-1}=W_{n,n-1}$. Such terms describe the increment between iterates rather than the accumulated
amount of each gradient. In these terms, the above operation is exactly the
anti-transpose operation. This operation is well-studied, being at the core of
the H-duality theory of~\cite{KOPR23}. For our analysis, cumulative matrices
$W$ rather than incremental matrices $H$ are more convenient, so we keep our
analysis in $W$ terms.

Our proof of optimality and uniqueness for OGM-G will closely mirror that of
OGM, up to a reversal in coordinate ordering in the reduced simplex problem.
This mirrors the reversal underlying the known H-duality relationship between
the two methods' convergence upper bounds~\cite{KOPR23}.

Let $\mathcal{F}(F)=1-(f_0-f_\star)$ denote the nonnegative normalization
constraint in this setting. We retain the same Gram representation as in the
suboptimality PEP: $G=P^\top P$ with
$P=[x_0-x_\star,g_0,\dots,g_N]$. For a fixed method $W$, its worst-case final
gradient norm can be equivalently expressed (assuming $d\geq N+2$) as
\begin{equation}\label{eq:ogmg-PEP-SDP}
\begin{aligned}
    \sup_{F,G}\quad &\frac12\|g_N\|^2\\
    \mathrm{s.t.}\quad
    &\mathcal Q_{i,j}(F,G;W)\geq0
    \qquad \forall i\neq j\in\mathcal I_\star,\\
    &\mathcal F(F)\geq0,\\
    &G\succeq0.
\end{aligned}
\end{equation}
Dually, this can be approached by considering nonnegative off-diagonal
multipliers $\lambda_{i,j}$, a nonnegative multiplier $\sigma$, and a positive
semidefinite matrix $Z$. To prove an upper bound $r$, these multipliers must
satisfy the identity
\begin{equation}\label{eq:ogmg-main-identity}
    r-\frac12\|g_N\|^2
    =\sum_{i\neq j\in\mathcal I_\star}\lambda_{i,j}
      \mathcal Q_{i,j}(F,G;W)
      +\sigma\mathcal F(F)+\langle Z,G\rangle
\end{equation}
for all vectors $F=(f_0-f_\star,\dots,f_N-f_\star)$ and symmetric matrices $G$.

As before, the diagonal interpolation terms vanish identically. So we have freedom to impose a convention for them. In this
PEP for the squared gradient norm, we use the mirrored convention
\begin{equation}\label{eq:gradient-lambda-diagonal}
    \lambda_{0,0}
    =-r
     -\sum_{j\in\mathcal I_\star\setminus\{0\}}\lambda_{j,0}, \qquad \lambda_{i,i}
    =
     -\sum_{j\in\mathcal I_\star\setminus\{i\}}\lambda_{j,i}
    \qquad \forall i=1,\dots,N.
\end{equation}
The additional term at $i=0$ reflects the normalization
$\mathcal F(F)=1-(f_0-f_\star)$.
Let $\Lambda=(\lambda_{i,j})_{i,j\in\mathcal I}$. Solving the identity
constraint gives the following compact dual. Its proof is deferred to
Appendix~\ref{app:gradient-dual}.
\begin{proposition}\label{prop:gradient-dual}
For any unit lower triangular matrix $W$, fixing $L=1$, the
PEP~\eqref{eq:ogmg-PEP-SDP} has dual
\begin{equation}\label{eq:ogmg-PEP-SDP-dual-clean}
\begin{aligned}
    \inf_{\Lambda,r}\quad &r\\
    \mathrm{s.t.}\quad
    &\Lambda_{i,j}\geq0
    \qquad \forall i\neq j\in\mathcal I,\\
    &\Lambda\mathbf1\leq0,
    \qquad
    \Lambda^\top\mathbf1=-r e_0,\\
    &-\operatorname{sym}(\Lambda^\top W)
      -\frac r2e_0e_0^\top
      -\frac12e_Ne_N^\top
      \succeq0.
\end{aligned}
\end{equation}
\end{proposition}

\noindent The following proposition establishes strong duality between this primal--dual pair, proven in Appendix~\ref{app:gradient-duality}.
\begin{proposition}\label{prop:gradient-duality}
For any method $W$, the primal~\eqref{eq:ogmg-PEP-SDP} and the dual
PEP~\eqref{eq:ogmg-PEP-SDP-dual-clean} have equal value.
\end{proposition}

\subsection{A Direct Triangular Reformulation}
We first rewrite the transpose of the compact dual multiplier as $-r^{-1}\Lambda^\top=TC$.
Note this exists since the factorization in Proposition~\ref{prop:M-matrix-factorization} applies even
when $\Lambda$ is singular. Combining it with
Proposition~\ref{prop:gradient-dual} gives the following reformulation.

\begin{lemma}[The reversed $TC$ reformulation]\label{lem:ogmg-TC}
The value in~\eqref{eq:subopttograd-design} is
\begin{equation}\label{eq:ogmg-TC-problem}
\begin{aligned}
 r_{F\rightarrow G}
 =\inf_{W,T,C,r}\quad &r\\
 \mathrm{s.t.}\quad
 &W\text{ is unit lower triangular},\\
 &T\text{ is upper triangular},\qquad T_{i,i}\geq0,\qquad T_{i,j}\leq0\quad \forall i<j,
 \qquad T_{0,0}=1,\\
 &C\text{ is unit lower triangular},\qquad C_{i,j}\leq0\quad \forall i>j,
 \qquad C\mathbf1=e_0,\\
 &(TC)_{i,j}\leq0\quad \forall i\neq j,\qquad C^\top T^\top\mathbf1\geq0,\\
 &\operatorname{sym}(TCW)
   -\frac12e_0e_0^\top
   -\frac1{2r}e_Ne_N^\top
   \succeq0,
 \qquad r>0.
\end{aligned}
\end{equation}
For each fixed $W$, factoring $-r^{-1}\Lambda^\top=TC$ identifies the feasible dual PEP solutions with those of the above reformulation.
\end{lemma}

\begin{proof}
Let $(W,\Lambda,r)$ be a feasible algorithm and PEP proof pair, and set
$A=-r^{-1}\Lambda^\top$. Then $A$ has nonpositive off-diagonal entries,
\[
    A\mathbf1=e_0,
    \qquad
    \mathbf1^\top A\geq0.
\]
Proposition~\ref{prop:M-matrix-factorization} gives a factorization $A=TC$
with the stated triangular forms and signs, chosen so that
$C\mathbf1=e_0$ and $T_{0,0}=1$. Substituting
$\Lambda=-rC^\top T^\top$ into the semidefinite constraint gives
\[
r\left(
\operatorname{sym}(TCW)
-\frac12e_0e_0^\top
-\frac1{2r}e_Ne_N^\top
\right)\succeq0,
\]
and the remaining scalar constraints become exactly those in
\eqref{eq:ogmg-TC-problem}.
Conversely, every feasible tuple $(W,T,C,r)$ gives a feasible compact dual
tuple by setting $\Lambda=-rC^\top T^\top$.
\end{proof}

For fixed $T,C$, the semidefinite constraint is again a triangular PSD
completion.

\begin{lemma}[Eliminating $W$]\label{lem:ogmg-eliminate-W}
Fix admissible $T,C$ and $r>0$. There exists a unit lower triangular $W$
satisfying the semidefinite constraint in~\eqref{eq:ogmg-TC-problem} if and
only if there is a vector $b\geq0$ satisfying
\begin{equation}\label{eq:ogmg-backward-solution}
    Tb=e_N,
    \qquad
    b_i=0\quad\text{whenever }T_{i,i}=0
\end{equation}
and
\begin{equation}\label{eq:ogmg-diagonal-conditions}
(b_0)^2\leq r,
\qquad
\frac{T_{i,i}}{2}b_i^2\leq r
\qquad \forall 1\leq i\leq N.
\end{equation}
If every inequality is tight, then this $W$ is unique.
\end{lemma}

\begin{proof}
If $\operatorname{sym}(TCW)-\frac12e_0e_0^\top
-\frac1{2r}e_Ne_N^\top \succeq0$ and $y\in\ker T^\top$, then
$$0\leq y^\top \left(\operatorname{sym}(TCW)-\frac12e_0e_0^\top
-\frac1{2r}e_Ne_N^\top\right)y=-y_0^2/2-y_N^2/(2r),$$
so $e_N\in\operatorname{range}(T)$.
Moreover, since $C^{-\top}\geq0$ and $C^\top T^\top\mathbf1\geq0$,
\[
    T^\top\mathbf1
    =C^{-\top}(C^\top T^\top\mathbf1)\geq0.
\]
Thus $T_{i,i}=0$ forces the $i$th column of $T$ to vanish. Since
$e_N\in\operatorname{range}(T)$, backward substitution gives a solution of
\eqref{eq:ogmg-backward-solution}: when $T_{i,i}>0$, the nonpositive
entries above the diagonal make the resulting coordinate nonnegative, while when
$T_{i,i}=0$, the coordinate is free and is set equal to zero. Thus $b\geq0$.
For $T_{i,i}>0$, solve $T^\top x=e_i$ by forward substitution. Then
$x_i=1/T_{i,i}$, 
$$x_N=x^\top e_N=x^\top Tb=e_i^\top b=b_i,$$
and $x^\top TCWx=1/T_{i,i}$. Hence, we have that
$$ 0 \leq x^\top\left(\operatorname{sym}(TCW)-\frac12e_0e_0^\top
-\frac1{2r}e_Ne_N^\top\right)x = \frac{1}{T_{i,i}} - \frac{1}{2}x_0^2 - \frac{1}{2r}b_i^2.$$
For $i=0$, since $T_{0,0}=1$ and $x_0=1$, this implies $b_0^2\leq r$, while for $i\geq 1$, since $x_0=0$, this gives $\frac{T_{i,i}}{2}b_i^2\leq r$. That is,~\eqref{eq:ogmg-diagonal-conditions}.

Conversely, take such a vector $b$ and define
the lower triangular matrix $L$ by
\[
 L_{i,i}=\begin{cases}T_{i,i}^{-1},&T_{i,i}>0,\\0,&T_{i,i}=0,\end{cases}
 \qquad
 L_{i,j}=\frac1r b_ib_j\quad(i>j).
\]
Then $\operatorname{sym}(L)-\frac12e_0e_0^\top
 -\frac1{2r}bb^\top\succeq0$.
Let $X$ be unit lower triangular, with its $i$th row equal to the $i$th row
of $LT^\top$ when $T_{i,i}>0$ and equal to $e_i^\top$ when $T_{i,i}=0$.
Since every column of $T$ with zero diagonal entry vanishes,
$TX=TLT^\top$. Thus $W=C^{-1}X$ is unit lower triangular and
\[
 \operatorname{sym}(TCW)-\frac12e_0e_0^\top
 -\frac1{2r}e_Ne_N^\top=T\left(\operatorname{sym}(L)-\frac12e_0e_0^\top
 -\frac1{2r}bb^\top\right)T^\top\succeq0.
\]
If every inequality is tight, then $T$ is nonsingular because $r>0$; the
usual congruence by $T^{-\top}$ gives a positive semidefinite matrix with
zero diagonal. Hence this slack matrix must be exactly zero, which uniquely determines $W$.
\end{proof}

\subsection{Projection onto the Simplex Problem}

Ultimately, we will solve~\eqref{eq:ogmg-TC-problem}, finding the unique
optimal factors $T_G^\star,C_G^\star$ and optimal value
$r_\mathtt{OGM}$. Namely, let $T_G^\star$ be the upper triangular matrix
with entries
\begin{equation}\label{eq:ogmg-T-star-r}
\begin{aligned}
(T_G^\star)_{0,0}=1,
\qquad
(T_G^\star)_{i,i}=\frac{1}{2r_\mathtt{OGM}\theta_{N-i}^2}
\quad \forall 1\leq i\leq N,
\qquad
(T_G^\star)_{i,j}
=-\bigl((T_G^\star)_{i+1,i+1}-(T_G^\star)_{i,i}\bigr)
\quad \forall i<j,
\end{aligned}
\end{equation}
and let $C_G^\star$ be the unit lower triangular matrix with
\begin{equation}\label{eq:ogmg-C-star-r}
(C_G^\star)_{i,i}=1,
\qquad
(C_G^\star)_{i,i-1}=-1
\qquad \forall 1\leq i\leq N,
\end{equation}
and all other entries below the diagonal equal to zero. The following lemma
establishes that $r_\mathtt{OGM}$ is a lower bound on
\eqref{eq:ogmg-TC-problem} and that, if it is attained, the associated
factors must equal $T_G^\star,C_G^\star$. It also shows that every condition
in~\eqref{eq:ogmg-diagonal-conditions} is then tight, so
Lemma~\ref{lem:ogmg-eliminate-W} will uniquely determine $W$. Doing so, one can (re)discover OGM-G, given in Appendix~\ref{subsec:discovering-ogmg}.

\begin{lemma}[Reversed Projection onto the Simplex Problem]
\label{lem:ogmg-TC-projection}
Let $W,T,C,r$ be feasible in~\eqref{eq:ogmg-TC-problem}, let $b$ be chosen as
in~\eqref{eq:ogmg-backward-solution}, and set
\begin{equation}\label{eq:ogmg-green-quantities}
q=C^{-\top}e_N,
\qquad
D_i=\sum_{j=i}^N b_jq_j
\quad \forall 0\leq i\leq N.
\end{equation}
Define $\delta\in\mathbb{R}^{N+1}$ by
\begin{equation}\label{eq:ogmg-delta-from-TC}
\delta_{N-i}=\frac{b_iq_i}{D_0}
\qquad \forall 0\leq i\leq N.
\end{equation}
Then $D_0\geq1$, $\delta\in\Delta_{N+1}$, and
\begin{equation}\label{eq:ogmg-projected-chain}
\max\left\{
\frac{\delta_i^2}{2\sum_{j=0}^i\delta_j}:0\leq i<N,
\ \delta_N^2
\right\}
\leq r.
\end{equation}
Consequently, $r\geq r_\mathtt{OGM}$. If $r=r_\mathtt{OGM}$, then
$T=T_G^\star$, $C=C_G^\star$, and every inequality in
\eqref{eq:ogmg-diagonal-conditions} is tight.
\end{lemma}

\begin{proof}
The normalization $C\mathbf1=e_0$ gives $C^{-1}e_0=\mathbf1$, and therefore
$q_0=e_N^\top C^{-1}e_0=1$. Also $q\geq0$.
Fix $i$ with $T_{i,i}>0$, and consider the trailing principal block
$A_i=(TC)_{i:N,i:N}$. The vectors
\[
u=\frac1{T_{i,i}}C_{i:N,i:N}^{-1}e_i,
\qquad
v=C_{i:N,i:N}^{-1}b_{i:N}
\]
satisfy $A_iu=e_i$ and $A_iv=e_N$.
Their relevant coordinates are
\[
u_i=\frac1{T_{i,i}},
\qquad
u_N=\frac{q_i}{T_{i,i}},
\qquad
v_N=D_i.
\]
Lemma~\ref{lem:app-inverse-bound}, applied to this principal block, gives
\begin{equation}\label{eq:ogmg-singular-inverse-bound}
q_i\leq1,
\qquad
q_i^2\leq T_{i,i}D_i.
\end{equation}
If $T_{i,i}=0$, then $b_i=0$ by~\eqref{eq:ogmg-backward-solution}, so the
corresponding coordinate $\delta_{N-i}$ vanishes. The case $i=0$ gives
$D_0\geq1$.

All terms in~\eqref{eq:ogmg-delta-from-TC} are nonnegative and their sum is
one. Moreover, the partial sums are given by $\sum_{\ell=0}^{N-i}\delta_\ell= D_i/D_0$.
The terminal coordinate then satisfies $\delta_N^2= b_0^2/D_0^2\leq r$.
For every $1\leq i\leq N$ with $T_{i,i}>0$, equations
\eqref{eq:ogmg-diagonal-conditions} and
\eqref{eq:ogmg-singular-inverse-bound} give
\[
\frac{\delta_{N-i}^2}
{2\sum_{\ell=0}^{N-i}\delta_\ell}
=
\frac{b_i^2q_i^2}{2D_iD_0}
\leq
\frac{T_{i,i}b_i^2}{2D_0}
\leq r.
\]
If $T_{i,i}=0$, then $\delta_{N-i}=0$, so the same bound holds.
This proves~\eqref{eq:ogmg-projected-chain}, and
Lemma~\ref{lem:ogm-simplex-prob} gives $r\geq r_\mathtt{OGM}$.

Suppose now that $r=r_\mathtt{OGM}$. Then
$\delta=\delta^\star$ and every term in the preceding chains is tight. Since
every coordinate of $\delta^\star$ is positive, every $b_i$ is positive.
By~\eqref{eq:ogmg-backward-solution}, $T_{i,i}>0$ for every $i$, so $T$ is
nonsingular. Tightness of the
terminal chain gives $D_0=1$ and $b_0=\sqrt{r_\mathtt{OGM}}$.
For every $i\geq1$, equality in~\eqref{eq:ogmg-singular-inverse-bound}
gives $q_i^2=T_{i,i}D_i$. Lemma~\ref{lem:app-inverse-bound} gives
\[
    \frac{q_i}{T_{i,i}}
    \leq\min\left\{\frac1{T_{i,i}},D_i\right\},
    \qquad
    \left(\frac{q_i}{T_{i,i}}\right)^2=\frac{D_i}{T_{i,i}}.
\]
Thus
\[
    \frac{q_i}{T_{i,i}}=\frac1{T_{i,i}}=D_i,
    \qquad
    q_i=1
    \qquad \forall 0\leq i\leq N.
\]
Equivalently, $e_N^\top C^{-1}=\mathbf1^\top$, or
$\mathbf1^\top C=e_N^\top$. Together with $C\mathbf1=e_0$ and the nonpositive
entries below the diagonal, these two normalizations force~\eqref{eq:ogmg-C-star-r}:
row $1$ forces $C_{1,0}=-1$, column $0$ then has no other nonzero entry below the diagonal,
and induction gives the remaining adjacent entries.

Since $D_0=1$ and $q=\mathbf1$, equation~\eqref{eq:ogmg-delta-from-TC} gives
\[
b_i=\delta_{N-i}^\star
\qquad \forall 0\leq i\leq N.
\]
Tightness in~\eqref{eq:ogmg-diagonal-conditions} now fixes the diagonal of
$T$ as in~\eqref{eq:ogmg-T-star-r}. It remains only to determine the
entries above the diagonal. Since $C$ is the unit lower bidiagonal matrix with subdiagonal entries $-1$,
$C^\top T^\top\mathbf1\geq0$ and $T_{0,0}=1$ imply
\[
1=(\mathbf1^\top T)_0\geq(\mathbf1^\top T)_1\geq\cdots
\geq(\mathbf1^\top T)_N\geq0.
\]
Moreover, the entries $b_i=\delta_{N-i}^\star$ form a positive probability vector, and $(\mathbf1^\top T)b=\mathbf1^\top e_N=1$.
It follows that $\mathbf1^\top T=\mathbf1^\top$. The off-diagonal
inequalities for $TC$ give
\[
T_{i,j}\leq T_{i,j+1}
\qquad \forall 0\leq i<j<N.
\]
For every $1\leq i\leq N$, summing $Tb=e_N$ over rows $i,\dots,N$
gives
\[
    \sum_{j=i}^N b_j\sum_{k=i}^jT_{k,j}=1.
\]
On the other hand, equality in~\eqref{eq:ogmg-singular-inverse-bound}, together
with $q=\mathbf1$, gives
\[
    T_{i,i}\sum_{j=i}^N b_j=1.
\]
Subtracting these identities and using $\mathbf1^\top T=\mathbf1^\top$
yields $\sum_{j=i}^N b_j
\sum_{k=0}^{i-1}(T_{ki}-T_{kj})=0$.
Every weight is positive, and the off-diagonal inequalities above make every
summand nonpositive. Hence
\[
T_{kj}=T_{ki}
\qquad \forall j\geq i,\quad 0\leq k<i.
\]
Taking $i=k+1$ shows that the entries above the diagonal in each row of $T$ are equal. Subtracting
the column-sum identity in column $j$ from that in column $j+1$ gives
$T_{j,j+1}=T_{jj}-T_{j+1,j+1}$, proving the formula for the entries above the diagonal in
\eqref{eq:ogmg-T-star-r}. Thus $T=T_G^\star$ and $C=C_G^\star$.
Reading the equality chain backward also shows that every condition in
\eqref{eq:ogmg-diagonal-conditions} is tight.
\end{proof}

Just as occurred for OGM, these optimal factors $T_G^\star$ and $C_G^\star$ are also the unique limit of any sequence of factors converging towards optimality. Appendix~\ref{app:limiting-factors} proves this stability extension.
\begin{proposition}[Uniqueness of Limiting Optimal Factors]
\label{prop:ogmg-limiting-factors}
Let $W$ be fixed. If $(W,T^{(k)},C^{(k)},r^{(k)})$ is feasible in
\eqref{eq:ogmg-TC-problem} and $r^{(k)}\to r_\mathtt{OGM}$, then $T^{(k)}\to T_G^\star$ and $C^{(k)}\to C_G^\star$.
\end{proposition}

\subsection{Proof of Theorem~\ref{thm:subopttograd}}
Fix a method $W$, and let $v_{F\rightarrow G}(W)$ denote the exact value of its gradient-norm
PEP~\eqref{eq:ogmg-PEP-SDP}. By Proposition~\ref{prop:gradient-duality}, there is a sequence
$(\Lambda^{(k)},r^{(k)})$ feasible in~\eqref{eq:ogmg-PEP-SDP-dual-clean} for
this fixed $W$, with $r^{(k)}>0$ and $r^{(k)}\to v_{F\rightarrow G}(W)$. Factor
\[
    -\frac1{r^{(k)}}(\Lambda^{(k)})^\top=T^{(k)}C^{(k)}
\]
using Proposition~\ref{prop:M-matrix-factorization}. Lemma~\ref{lem:ogmg-TC}
makes $(W,T^{(k)},C^{(k)},r^{(k)})$ feasible in
\eqref{eq:ogmg-TC-problem}, and Lemma~\ref{lem:ogmg-TC-projection} gives
$r^{(k)}\geq r_\mathtt{OGM}$. Passing to the limit shows $v_{F\rightarrow G}(W)\geq r_\mathtt{OGM}$.

Suppose equality holds. Proposition~\ref{prop:ogmg-limiting-factors} gives
\[
    T^{(k)}\to T_G^\star,
    \qquad
    C^{(k)}\to C_G^\star.
\]
Passing to the limit in the semidefinite constraint of
\eqref{eq:ogmg-TC-problem} shows that
$(W,T_G^\star,C_G^\star,r_\mathtt{OGM})$ is feasible. The matrix
$T_G^\star$ is nonsingular, and every inequality in
\eqref{eq:ogmg-diagonal-conditions} is tight. Hence
Lemma~\ref{lem:ogmg-eliminate-W} shows that $W$ is the unique unit lower
triangular matrix satisfying the semidefinite constraint for these factors.

The guarantee in~\cite{KF21} gives $v_{F\rightarrow G}(W_\mathtt{OGM-G})\leq r_\mathtt{OGM}$.
Together with the preceding lower bound, this gives
$v_{F\rightarrow G}(W_\mathtt{OGM-G})=r_\mathtt{OGM}$. Moreover, every optimal $W$ must equal $W_\mathtt{OGM-G}$.

\section{Uniqueness and Optimality of Lemniscate Acceleration}\label{sec:lemni}

Finally, we consider minimizing the final squared gradient norm relative to the initial squared distance to a minimizer, i.e., $r_{D\rightarrow G}$ from~\eqref{eq:disttograd-design}. A value $r$ in~\eqref{eq:disttograd-design} is equivalent to the guarantee
\[
    \|\nabla f(x_N)\|
    \leq \sqrt r\,L\|x_0-x_\star\|.
\]
Kim, Ryu, and Das Gupta~\cite{KimRyuDasGupta26} recently introduced the Lemniscate acceleration method for this criterion. We show below that the method is exactly optimal among fixed-step first-order methods.

For a given budget $N\geq1$, let $\Omega_N>0$ and
\[
    1=\rho_0>\rho_1>\cdots>\rho_N>\rho_{N+1}=0
\]
be the unique values satisfying
\begin{equation}\label{eq:lemni-rate-recursion}
    \Omega_N(\rho_i-\rho_{i+1})^2
    =\rho_i(1-\rho_{i+1}^2)
    \qquad \forall 0\leq i\leq N.
\end{equation}
The existence and uniqueness of these values are established in~\cite{KimRyuDasGupta26}. For $0\leq i\leq N$, write
\begin{equation*}
    \phi_i=\frac{1+\rho_i^2}{2\rho_i}.
\end{equation*}
In particular, $\phi_0=1$. In our cumulative notation, Lemniscate acceleration is the unit lower triangular matrix $W_\mathtt{Lemni}$ given by
\begin{equation}\label{eq:lemni-W}
    (W_\mathtt{Lemni})_{n,n}=1,
    \qquad
    (W_\mathtt{Lemni})_{n,i}
    =1+\Omega_N(\phi_{i+1}-\phi_i)
      (\phi_{N-i}-\phi_{N-n})
    \quad \forall 0\leq i<n\leq N.
\end{equation}
Like OGM and OGM-G, this method also possesses a recursive description in terms of two sequences $x_n,z_n$. This equivalent description is given in~\cite{KimRyuDasGupta26}. For our development here, the matrix $W$ is the only description needed.

The convergence guarantee in~\cite{KimRyuDasGupta26} ensures that $\|g_N\|^2\leq L^2\|x_0-x_\star\|^2/\Omega_N^2$.
Accordingly, we denote $r_\mathtt{Lemni}=1/\Omega_N^2$.
Our goal is to show that this upper bound is the optimal value of~\eqref{eq:disttograd-design}.

\subsection{The PEP for Distance-to-Gradient Guarantees}

Let $\mathcal D(G)=1-\frac12\|x_0-x_\star\|^2$ be the same normalization used in the suboptimality PEP. For a fixed method $W$, its worst-case final gradient norm under an initial-distance bound can be equivalently expressed (assuming $d\geq N+2$) as
\begin{equation}\label{eq:lemni-PEP-SDP}
\begin{aligned}
    \sup_{F,G}\quad &\frac12\|g_N\|^2\\
    \mathrm{s.t.}\quad
    &\mathcal Q_{i,j}(F,G;W)\geq0
    \qquad \forall i\neq j\in\mathcal I_\star,\\
    &\mathcal D(G)\geq0,\\
    &G\succeq0.
\end{aligned}
\end{equation}
Dually, to prove an upper bound $r$, consider nonnegative off-diagonal
multipliers $\lambda_{i,j}$, a nonnegative multiplier $\sigma$, and a positive
semidefinite matrix $Z$ satisfying
\begin{equation}\label{eq:lemni-main-identity}
    r-\frac12\|g_N\|^2
    =\sum_{i\neq j\in\mathcal I_\star}\lambda_{i,j}
      \mathcal Q_{i,j}(F,G;W)
      +\sigma\mathcal D(G)+\langle Z,G\rangle
\end{equation}
for all vectors $F=(f_0-f_\star,\dots,f_N-f_\star)$ and symmetric matrices $G$.

As in the suboptimality setting, we use the diagonal convention
\begin{equation}\label{eq:lemni-lambda-diagonal}
    \lambda_{i,i}
    =-\sum_{j\in\mathcal I_\star\setminus\{i\}}\lambda_{j,i}
    \qquad \forall i\in\mathcal I.
\end{equation}
Let $\Lambda=(\lambda_{i,j})_{i,j\in\mathcal I}$. Solving the identity
constraint gives the following compact dual. Its proof is deferred to
Appendix~\ref{app:lemni-dual}.
\begin{proposition}\label{prop:lemni-dual}
For any unit lower triangular matrix $W$, fixing $L=1$, the
PEP~\eqref{eq:lemni-PEP-SDP} has dual
\begin{equation}\label{eq:lemni-PEP-SDP-dual-clean}
\begin{aligned}
    \inf_{\Lambda,r}\quad &r\\
    \mathrm{s.t.}\quad
    &\Lambda_{i,j}\geq0
    \qquad \forall i\neq j\in\mathcal I,\\
    &\Lambda^\top\mathbf1\leq0,
    \qquad
    \Lambda\mathbf1\leq0,\\
    &\begin{pmatrix}
        \frac r2 & \frac12\mathbf1^\top\Lambda\\[1mm]
        \frac12\Lambda^\top\mathbf1
        &-\operatorname{sym}(\Lambda^\top W)-\frac12e_Ne_N^\top
    \end{pmatrix}\succeq0.
\end{aligned}
\end{equation}
\end{proposition}
\noindent Strong duality holds between this primal--dual pair as well, again proven in Appendix~\ref{app:lemni-duality}.
\begin{proposition}\label{prop:lemni-duality}
For any method $W$, the primal~\eqref{eq:lemni-PEP-SDP} and the dual
PEP~\eqref{eq:lemni-PEP-SDP-dual-clean} have equal value.
\end{proposition}

\subsection{A Direct Triangular Reformulation}

We rewrite the scaled negated compact dual multiplier as
$-r^{-1/2}\Lambda=TC$.
The difference from the suboptimality setting is only in the row-sum
condition: because no function value appears in either the objective or the
normalization, neither side of the multiplier matrix has a distinguished
terminal source.

\begin{lemma}[The $TC$ reformulation]\label{lem:lemni-TC}
The value in~\eqref{eq:disttograd-design} is
\begin{equation}\label{eq:lemni-TC-problem}
\begin{aligned}
 r_{D\rightarrow G}
 =\inf_{W,T,C,r}\quad &r\\
 \mathrm{s.t.}\quad
 &W\text{ is unit lower triangular},\\
 &T\text{ is upper triangular},\qquad T_{i,i}\geq0,\qquad T_{i,j}\leq0 \quad\forall i<j,\\
 &C\text{ is unit lower triangular},\qquad C_{i,j}\leq0\quad \forall i>j,\\
 &(TC)_{i,j}\leq0\quad \forall i\neq j,\qquad TC\mathbf1\geq0,
 \qquad C^\top T^\top\mathbf1\geq0,\\
 &\begin{pmatrix}
       \frac{\sqrt r}{2} &-\frac12\mathbf1^\top TC\\[1mm]
       -\frac12C^\top T^\top\mathbf1
       &\operatorname{sym}(W^\top TC)-\frac1{2\sqrt r}e_Ne_N^\top
   \end{pmatrix}\succeq0,\qquad r>0.
\end{aligned}
\end{equation}
For each fixed $W$, factoring $-r^{-1/2}\Lambda=TC$ identifies the feasible dual PEP solutions with those of the above reformulation.
\end{lemma}

\begin{proof}
Let $(W,\Lambda,r)$ be a feasible algorithm and dual PEP certificate pair, and
set $A=-r^{-1/2}\Lambda$. Then $A$ has nonpositive off-diagonal entries, with
\[
    A\mathbf1\geq0,
    \qquad
    \mathbf1^\top A\geq0.
\]
Proposition~\ref{prop:M-matrix-factorization} gives a factorization $A=TC$
with the stated triangular forms and signs. Substituting
$\Lambda=-\sqrt r\,TC$ into the scalar constraints and dividing the
semidefinite constraint by $\sqrt r$ gives exactly the remaining constraints
in~\eqref{eq:lemni-TC-problem}. Conversely, given a feasible tuple
$(W,T,C,r)$, setting $\Lambda=-\sqrt r\,TC$ gives a feasible tuple.
\end{proof}

The positive semidefinite constraint is again a triangular PSD completion.

\begin{lemma}[Eliminating $W$]\label{lem:lemni-eliminate-W}
Fix admissible $T,C$ and $r>0$. There exists a unit lower triangular $W$
satisfying the semidefinite constraint in~\eqref{eq:lemni-TC-problem} if and
only if
\begin{equation}\label{eq:lemni-diagonal-conditions}
    (\mathbf1^\top Te_i)^2+(e_N^\top C^{-1}e_i)^2
    \leq 2\sqrt r\,T_{i,i}
    \qquad \forall 0\leq i\leq N.
\end{equation}
If $T$ is nonsingular and every inequality is tight, then this $W$ is unique.
\end{lemma}

\begin{proof}
Taking the Schur complement and congruence by $C^{-1}$ gives
\[
 \operatorname{sym}(T^\top WC^{-1})
 -\frac1{2\sqrt r}qq^\top-\frac1{2\sqrt r}pp^\top\succeq0,
 \qquad q=C^{-\top}e_N,\quad p=T^\top\mathbf1.
\]
Its diagonal gives~\eqref{eq:lemni-diagonal-conditions}. Conversely, the
triangular completion argument in Lemma~\ref{lem:ogm-eliminate-W} applies
with entries $(q_iq_j+p_ip_j)/\sqrt r$ below the diagonal. In the tight case, this positive semidefinite matrix is forced to have zero diagonal and hence be exactly the zero matrix. This being zero forces a unique value of the lower triangular matrix $T^\top WC^{-1}$ and hence $W$.
\end{proof}

\subsection{Projection onto the Lemniscate Problem}

The optimization problem in scalar variables that replaces the simplex problem
of the preceding sections is the following path problem. The convention $0/0=0$ is used at repeated
endpoints.

\begin{lemma}[The Lemniscate Problem]\label{lem:lemni-scalar-prob}
For every $N\geq1$, the optimization problem
\begin{equation}\label{eq:lemni-scalar-prob}
\begin{aligned}
 \min_{1=\rho_0\geq\rho_1\geq\cdots\geq\rho_{N+1}=0}\quad
 &\max_{0\leq i\leq N}
 \frac{(\rho_i-\rho_{i+1})^2}
 {\rho_i(1-\rho_{i+1}^2)}
\end{aligned}
\end{equation}
has optimal value $1/\Omega_N$, attained uniquely by the sequence in
\eqref{eq:lemni-rate-recursion}. At this sequence, every term in the finite
maximum equals $1/\Omega_N$. 
\end{lemma}

\begin{proof}
For the sequence in~\eqref{eq:lemni-rate-recursion}, every term in
\eqref{eq:lemni-scalar-prob} is exactly $1/\Omega_N$, so the displayed value is
attained.

For $0\leq b\leq a\leq1$, write
\[
    \Psi(a,b)=\frac{(a-b)^2}{a(1-b^2)},
\]
with $\Psi(0,0)=\Psi(1,1)=0$. For fixed $a>0$, this quantity is
strictly decreasing in $b$, and for fixed $b<a$, it is strictly increasing in
$a$. Given $m>0$, let $\psi_m(a)$ be the smallest $b\in[0,a]$ satisfying
$\Psi(a,b)\leq m$. Thus $\psi_m(a)=0$ when $a\leq m$, while otherwise it is
the unique point satisfying $\Psi(a,\psi_m(a))=m$. The preceding monotonicity
shows that $\psi_m$ is nondecreasing in $a$ and nonincreasing in $m$, strictly
so whenever its value is positive.

Let $m$ be the objective value of any feasible chain and define its lower
envelope by
\[
    \bar\rho_0=1,
    \qquad
    \bar\rho_{i+1}=\psi_m(\bar\rho_i)
    \quad(0\leq i\leq N).
\]
Since $\Psi(\rho_i,\rho_{i+1})\leq m$, induction gives
$\rho_i\geq\bar\rho_i$ for every $i$. Put $m_\star=1/\Omega_N$. The sequence
in~\eqref{eq:lemni-rate-recursion} is exactly the envelope generated by
$\psi_{m_\star}$, and its last coordinate is zero. If $m<m_\star$, strict
monotonicity gives $\bar\rho_i>\rho_i^\star$ for every $i\geq1$ and, in
particular, $\bar\rho_{N+1}>0$. This contradicts
$0=\rho_{N+1}\geq\bar\rho_{N+1}$. Hence $m\geq m_\star$.

If $m=m_\star$, the same comparison gives $\rho_i\geq\rho_i^\star$. Any strict
inequality propagates strictly through the remaining envelope and again makes
$\rho_{N+1}>0$. Therefore the optimizing chain is unique and equals the
sequence in~\eqref{eq:lemni-rate-recursion}; every local term is then tight.
\end{proof}

Ultimately, we will solve~\eqref{eq:lemni-TC-problem}, finding the unique
optimal factors $T_L^\star,C_L^\star$ and optimal value
$r_\mathtt{Lemni}$. Namely, let $C_L^\star$ be the unit lower triangular matrix
with
\begin{equation}\label{eq:lemni-C-star}
    (C_L^\star)_{i,i}=1,
    \qquad
    (C_L^\star)_{i,j}=-\frac{\phi_{j+1}-\phi_j}{\phi_i}
    \quad \forall 0\leq j<i\leq N,
\end{equation}
and let $T_L^\star$ be the upper triangular matrix with the only
nonzero entries
\begin{equation}\label{eq:lemni-T-star}
\begin{aligned}
    (T_L^\star)_{i,i}
    &=\frac{1-\rho_{i+1}^2}{2\rho_{i+1}}
      \frac{\phi_i}{\phi_{i+1}}
      &&\forall 0\leq i<N,\\
    (T_L^\star)_{i,i+1}
    &=-\frac{1-\rho_{i+1}^2}{2\rho_{i+1}}
      &&\forall 0\leq i<N,\\
    (T_L^\star)_{N,N}&=\phi_N.
\end{aligned}
\end{equation}
The following lemma establishes that $r_\mathtt{Lemni}$ is a lower bound on
\eqref{eq:lemni-TC-problem} and that, if it is attained, the associated
factors must equal $T_L^\star,C_L^\star$. It also shows that every
condition in~\eqref{eq:lemni-diagonal-conditions} is then tight, so
Lemma~\ref{lem:lemni-eliminate-W} will uniquely determine $W$. Doing so, one can (re)discover the Lemniscate method, given in Appendix~\ref{subsec:discovering-lemni}.

\begin{lemma}[Projection onto the Lemniscate Problem]\label{lem:lemni-TC-projection}
Let $W,T,C,r$ be feasible in~\eqref{eq:lemni-TC-problem}. Then
\begin{equation}\label{eq:lemni-projection-bound}
    \max_{0\leq i\leq N}
    \frac{(\mathbf1^\top Te_i)^2
    +(e_N^\top C^{-1}e_i)^2}{2T_{i,i}}
    \geq \frac1{\Omega_N}.
\end{equation}
Here the ratio is defined as zero when $T_{i,i}=0$, since
\eqref{eq:lemni-diagonal-conditions} forces its numerator to vanish.
Consequently, $r\geq r_\mathtt{Lemni}$. If $r=r_\mathtt{Lemni}$, then
$T=T_L^\star$, $C=C_L^\star$, and every inequality in
\eqref{eq:lemni-diagonal-conditions} is tight.
\end{lemma}

\begin{proof}
Put
\[
    p=T^\top\mathbf1,
    \qquad
    q=C^{-\top}e_N.
\]
Since $C^{-1}\geq0$ and $C^\top p\geq0$, we have $p\geq0$; also $q\geq0$.
If $T_{i,i}=0$, then the $i$th diagonal condition in
\eqref{eq:lemni-diagonal-conditions} gives $p_i=q_i=0$. As in the preceding
sections, nonnegative column sums imply that the entire $i$th column of $T$ is
zero. The columns indexed by $i$ with $T_{i,i}>0$ are linearly independent, so
\[
    \operatorname{range}(T^\top)
    =\{z:z_i=0\text{ whenever }T_{i,i}=0\}.
\]
Hence $q$ lies in the range of $T^\top$, and $C^\top q=e_N$ shows that
$e_N$ lies in the range of $(TC)^\top$. Lemma~\ref{lem:app-inverse-bound}
gives
\[
\operatorname{range}(TC)=\operatorname{range}((TC)^\top).
\]
Since $C$ is invertible, $e_N$ therefore lies in the range of $T$.
Backward substitution gives a vector $b\geq0$ satisfying $Tb=e_N$, where we
set $b_i=0$ whenever $T_{i,i}=0$.

Here $p,q,b\geq0$, and $q_N=1$. Let the left-hand side
of~\eqref{eq:lemni-projection-bound} be $m$. By
\eqref{eq:lemni-diagonal-conditions}, $m\leq\sqrt r$. Define
\[
    E_{-1}=0,
    \qquad
    E_i=\sum_{j=0}^i p_jb_j
    \quad \forall 0\leq i\leq N,
\]
and
\[
    D_i=\sum_{j=i}^N b_jq_j
    \quad \forall 0\leq i\leq N,
    \qquad
    D_{N+1}=0.
\]
Then
\[
    0=E_{-1}\leq E_0\leq\cdots\leq E_N=1,
    \qquad
    D_0\geq D_1\geq\cdots\geq D_N>D_{N+1}=0.
\]
The same cut expansion as~\eqref{eq:ogm-cut-r} gives
\begin{equation}\label{eq:lemni-forward-cut}
    T_{i,i}b_i\leq E_i
    \qquad \forall 0\leq i\leq N.
\end{equation}

For the reverse inequality, fix $i$ with $T_{i,i}>0$ and consider
the trailing principal block $A_i=(TC)_{i:N,i:N}$. The vectors
\[
u=C_{i:N,i:N}^{-1}b_{i:N},
\qquad
v=\frac1{T_{i,i}}C_{i:N,i:N}^{-1}e_i
\]
solve $A_iu=e_N$ and $A_iv=e_i$. Their relevant coordinates are
$u_i=b_i$, $u_N=D_i$, and $v_i=1/T_{i,i}$.
Applying Lemma~\ref{lem:app-inverse-bound} with the two source indices in the
order $N,i$ gives
$b_i\leq\min\{D_i,1/T_{i,i}\}$. When
$T_{i,i}=0$, our choice gives $b_i=0$, so in all cases
\begin{equation}\label{eq:lemni-reverse-cut}
    b_i\leq D_i
    \qquad \forall 0\leq i\leq N.
\end{equation}

For every $i$, the definition of $m$ gives $p_i^2+q_i^2\leq2T_{i,i}m$.
Multiplying by $b_i^2$ and using
\[
    E_i-E_{i-1}=p_ib_i,
    \qquad
    D_i-D_{i+1}=b_iq_i,
\]
together with~\eqref{eq:lemni-forward-cut}--\eqref{eq:lemni-reverse-cut},
yields
\begin{equation}\label{eq:lemni-path-inequality}
    \frac{(E_i-E_{i-1})^2+(D_i-D_{i+1})^2}
    {2E_iD_i}
    \leq m
    \qquad \forall 0\leq i\leq N.
\end{equation}
If $E_iD_i=0$, the preceding inequality forces the $i$th
step to have zero length (by our convention that $0/0=0$). Thus it suffices to focus only on indices with $E_iD_i>0$.

For $0\leq i\leq N+1$, write $(E_{i-1},D_i)=R_i(\cos\vartheta_i,\sin\vartheta_i)$.
The monotonicity above gives
\[
    \vartheta_0=\frac\pi2
    \geq\vartheta_1\geq\cdots\geq
    \vartheta_{N+1}=0.
\]
The $i$th numerator in~\eqref{eq:lemni-path-inequality} is the squared
distance between the consecutive points with indices $i$ and $i+1$, while its
denominator is
\[
    2E_iD_i
    =2R_iR_{i+1}\cos\vartheta_{i+1}\sin\vartheta_i.
\]
Consequently,
\[
\begin{aligned}
&\frac{(E_i-E_{i-1})^2+(D_i-D_{i+1})^2}{2E_iD_i}\\
&\qquad=
\frac{R_i/R_{i+1}+R_{i+1}/R_i
-2\cos(\vartheta_i-\vartheta_{i+1})}
{2\cos\vartheta_{i+1}\sin\vartheta_i}\\
&\qquad\geq
\frac{1-\cos(\vartheta_i-\vartheta_{i+1})}
{\cos\vartheta_{i+1}\sin\vartheta_i},
\end{aligned}
\]
where the final inequality uses
$R_i/R_{i+1}+R_{i+1}/R_i\geq2$.

Set
\[
    \rho_i=\tan\frac{\vartheta_i}{2}
    \qquad \forall 0\leq i\leq N+1.
\]
Then
\[
    1=\rho_0\geq\rho_1\geq\cdots\geq\rho_{N+1}=0,
\]
and the half-angle identities give
\[
    \frac{1-\cos(\vartheta_i-\vartheta_{i+1})}
    {\cos\vartheta_{i+1}\sin\vartheta_i}
    =
    \frac{(\rho_i-\rho_{i+1})^2}
    {\rho_i(1-\rho_{i+1}^2)}.
\]
Combining this identity with~\eqref{eq:lemni-path-inequality} and
Lemma~\ref{lem:lemni-scalar-prob} proves
\[
\sqrt r\geq m\geq\frac1{\Omega_N},
\]
which is~\eqref{eq:lemni-projection-bound} and gives
$r\geq r_\mathtt{Lemni}$.

Suppose now that $r=r_\mathtt{Lemni}$. The unique scalar optimizer has every
local term equal to $1/\Omega_N$. Since each such term is bounded above by the
corresponding path ratio, which is in turn bounded above by $m\leq\sqrt r$,
every inequality in the preceding chain is tight, and the scalar sequence is
the unique sequence in~\eqref{eq:lemni-rate-recursion}. Equality in the radial inequality gives
$R_i=R_{i+1}$ for every $i$. Since $R_0=D_0$ and
$R_{N+1}=E_N=1$, every radius equals one. The half-angle formulas therefore
give
\begin{equation}\label{eq:lemni-equality-path}
    E_{i-1}=\frac{1-\rho_i^2}{1+\rho_i^2},
    \qquad
    D_i=\frac{2\rho_i}{1+\rho_i^2}=\frac1{\phi_i}.
\end{equation}
Equality in~\eqref{eq:lemni-forward-cut}--\eqref{eq:lemni-reverse-cut} gives
\[
    b_i=D_i=\frac1{\phi_i},
    \qquad
    T_{i,i}b_i=E_i.
\]
In particular every $b_i$ is positive, so $T_{i,i}>0$ for every
$i$ and $T$ is nonsingular. These identities give the diagonal in
\eqref{eq:lemni-T-star}.

Equality in the forward cut forces every entry of $T$ above the first
superdiagonal to vanish. The equation $Tb=e_N$ then gives
\[
    T_{i,i+1}
    =-T_{i,i}\frac{b_i}{b_{i+1}}
    =-\frac{1-\rho_{i+1}^2}{2\rho_{i+1}},
\]
which completes~\eqref{eq:lemni-T-star}.

It remains to determine $C$. Put $y=C\mathbf1$. Since $T$ is a
nonsingular $M$-matrix and $Ty=TC\mathbf1\geq0$, one has
$y\geq0$. The upper bidiagonal form and $Tb=e_N$ show that
\[
    \frac{y_0}{b_0}\geq\frac{y_1}{b_1}\geq\cdots\geq\frac{y_N}{b_N}\geq0.
\]
Here $y_0=b_0=1$. Moreover, $q^\top y=e_N^\top\mathbf1=1$ and $q^\top b=D_0=b_0=1$.
Every $q_i b_i=D_i-D_{i+1}$ is positive, so the preceding monotonicity forces
$y=b$.

For $i>j$, set $\gamma_{i,j}=-C_{i,j}/b_i\geq0$. The identity
$C\mathbf1=b$ and the off-diagonal inequalities $(TC)_{i,j}\leq0$ give
\[
    \sum_{j<i}\gamma_{i,j}=\phi_i-1,
    \qquad
    \gamma_{i+1,j}\leq\gamma_{i,j}
    \quad \forall j<i<N.
\]
We prove by induction on $j$ that
$\gamma_{i,j}=\phi_{j+1}-\phi_j$ for every $i>j$. The row $i=1$ gives
$\gamma_{1,0}=\phi_1-1=\phi_1-\phi_0$. If the assertion is known for columns
below $j$, the row $j+1$ identity gives
$\gamma_{j+1,j}=\phi_{j+1}-\phi_j$, and monotonicity gives
$\gamma_{i,j}\leq\phi_{j+1}-\phi_j$ for all $i>j$. On the other hand,
$C^\top q=e_N$ and $D_i=b_i$ give, with $b_{N+1}=0$,
\[
    q_j=\sum_{i=j+1}^N\gamma_{i,j}(b_i-b_{i+1}),
    \qquad
    q_j=(\phi_{j+1}-\phi_j)b_{j+1},
    \qquad
    \sum_{i=j+1}^N(b_i-b_{i+1})=b_{j+1}.
\]
The strict $\rho$-sequence makes every weight positive, so equality forces
$\gamma_{i,j}=\phi_{j+1}-\phi_j$ for every $i>j$. Thus
\[
    C_{i,j}=-b_i(\phi_{j+1}-\phi_j)
    =-\frac{\phi_{j+1}-\phi_j}{\phi_i},
\]
which is~\eqref{eq:lemni-C-star}. Thus
$T=T_L^\star$ and $C=C_L^\star$. Every local objective
inequality is also tight, so every inequality in
\eqref{eq:lemni-diagonal-conditions} is tight.
\end{proof}

As with the previous settings, these optimal factors $T_L^\star$ and $C_L^\star$ are also the unique limit of any sequence of factors converging towards optimality. Appendix~\ref{app:limiting-factors} provides a proof of this stability result.

\begin{proposition}[Uniqueness of Limiting Optimal Factors]
\label{prop:lemni-limiting-factors}
Let $W$ be fixed. If $(W,T^{(k)},C^{(k)},r^{(k)})$ is feasible in
\eqref{eq:lemni-TC-problem} and $r^{(k)}\to r_\mathtt{Lemni}$, then $T^{(k)}\to T_L^\star$ and $C^{(k)}\to C_L^\star$.
\end{proposition}

\subsection{Proof of Theorem~\ref{thm:disttograd}}
Fix a method $W$, and let $v_{D\rightarrow G}(W)$ denote the exact value of its
distance-to-gradient PEP~\eqref{eq:lemni-PEP-SDP}. By Proposition~\ref{prop:lemni-duality}, there is a
sequence $(\Lambda^{(k)},r^{(k)})$ feasible in
\eqref{eq:lemni-PEP-SDP-dual-clean} for this fixed $W$, with $r^{(k)}>0$ and
$r^{(k)}\to v_{D\rightarrow G}(W)$. Factor
\[
    -\frac1{\sqrt{r^{(k)}}}\Lambda^{(k)}=T^{(k)}C^{(k)}
\]
using Proposition~\ref{prop:M-matrix-factorization}. Lemma~\ref{lem:lemni-TC}
makes $(W,T^{(k)},C^{(k)},r^{(k)})$ feasible in
\eqref{eq:lemni-TC-problem}, and Lemma~\ref{lem:lemni-TC-projection} gives
\[
    \sqrt{r^{(k)}}\geq\frac1{\Omega_N}.
\]
Passing to the limit shows $ v_{D\rightarrow G}(W)\geq\frac1{\Omega_N^2}=r_\mathtt{Lemni}$.

Suppose equality holds. Proposition~\ref{prop:lemni-limiting-factors} gives
\[
    T^{(k)}\to T_L^\star,
    \qquad
    C^{(k)}\to C_L^\star.
\]
Passing to the limit in the semidefinite constraint of
\eqref{eq:lemni-TC-problem} shows that
$(W,T_L^\star,C_L^\star,r_\mathtt{Lemni})$ is feasible. The matrix
$T_L^\star$ is nonsingular, and
every inequality in~\eqref{eq:lemni-diagonal-conditions} is tight. Hence
Lemma~\ref{lem:lemni-eliminate-W} shows that $W$ is the unique unit lower
triangular matrix satisfying the semidefinite constraint for these factors.

The guarantee in~\cite{KimRyuDasGupta26} gives $v_{D\rightarrow G}(W_\mathtt{Lemni})\leq r_\mathtt{Lemni}$.
Together with the preceding lower bound, this gives
$v_{D\rightarrow G}(W_\mathtt{Lemni})=r_\mathtt{Lemni}$. Moreover, every optimal $W$ must equal $W_\mathtt{Lemni}$.

\section{Conclusion}
Our main theorems establish the optimality and uniqueness of OGM, OGM-G, and Lemniscate methods for their respective settings among fixed-step first-order methods with $d\geq N+2$. Previously, only the optimality of OGM was known. Fundamental to each proof is the decomposition of the associated (scaled) PEP multiplier matrix into a product of upper and lower triangular matrices $TC$. In each setting, this common reusable approach enabled subsequent reductions and simplifications.

When reducing final suboptimality, the optimal $C$ factor is the identity matrix, making the resulting optimization concern only upper triangular dual multiplier (proof) matrices $\Lambda = -T$. For both considered settings reducing final squared gradient norms, the optimal $C$ matrix was more complex, leading to more complex proof structures. Regardless, in each case, the task of finding the optimal method reduced to a (relatively) simple optimization problem from which the optimal method, and its uniqueness, could be readily extracted.

Such decompositions of dual multiplier matrices may prove useful beyond the considered smooth convex setting. In particular, once the optimal $C$ matrix is identified, one can reduce the search for optimal algorithms and proofs to a simpler search over upper triangular matrices $T$. This is similar to the constructive approach of Drori and Taylor~\cite{DT20}, which gave a constructive algorithm design procedure for reducing suboptimality using only upper triangular dual multiplier matrices. Fixing our discovered $C_G^\star$ and $C_L^\star$ may enable new variants of the constructive approach for minimization of the squared gradient norm.

{\small \paragraph{Acknowledgment.} Benjamin Grimmer was supported in this work as a fellow of the Alfred P.~Sloan Foundation and by the Air Force Office of Scientific Research (AFOSR) under grant FA9550261B093. Chanwoo Park was supported by an Amazon AI PhD Fellowship and the Korea Foundation for Advanced Studies (KFAS).}

{
\small
\bibliographystyle{unsrt}
\bibliography{main}

@book{BermanPlemmons94,
  author    = {Berman, Abraham and Plemmons, Robert J.},
  title     = {Nonnegative Matrices in the Mathematical Sciences},
  series    = {Classics in Applied Mathematics},
  volume    = {9},
  publisher = {Society for Industrial and Applied Mathematics},
  address   = {Philadelphia, PA},
  year      = {1994},
  doi       = {10.1137/1.9781611971262},
  url       = {https://doi.org/10.1137/1.9781611971262}
}

@article{DGVR24,
  author  = {Das Gupta, Shuvomoy and Van Parys, Bart P. G. and Ryu, Ernest K.},
  title   = {Branch-and-Bound Performance Estimation Programming: A Unified Methodology for Constructing Optimal Optimization Methods},
  journal = {Mathematical Programming},
  volume  = {204},
  number  = {1--2},
  pages   = {567--639},
  year    = {2024},
  doi     = {10.1007/s10107-023-01973-1},
  url     = {https://doi.org/10.1007/s10107-023-01973-1}
}

@article{DR17,
  author  = {Drori, Yoel},
  title   = {The Exact Information-Based Complexity of Smooth Convex Minimization},
  journal = {Journal of Complexity},
  volume  = {39},
  pages   = {1--16},
  year    = {2017},
  doi     = {10.1016/j.jco.2016.11.001},
  url     = {https://doi.org/10.1016/j.jco.2016.11.001}
}

@article{DT14,
  author  = {Drori, Yoel and Teboulle, Marc},
  title   = {Performance of First-Order Methods for Smooth Convex Minimization: A Novel Approach},
  journal = {Mathematical Programming},
  volume  = {145},
  number  = {1--2},
  pages   = {451--482},
  year    = {2014},
  doi     = {10.1007/s10107-013-0653-0},
  url     = {https://doi.org/10.1007/s10107-013-0653-0}
}

@article{DT20,
  author  = {Drori, Yoel and Taylor, Adrien B.},
  title   = {Efficient First-Order Methods for Convex Minimization: A Constructive Approach},
  journal = {Mathematical Programming},
  volume  = {184},
  number  = {1--2},
  pages   = {183--220},
  year    = {2020},
  doi     = {10.1007/s10107-019-01410-2},
  url     = {https://doi.org/10.1007/s10107-019-01410-2}
}

@article{GSW26,
  author  = {Grimmer, Benjamin and Shu, Kevin and Wang, Alex L.},
  title   = {Beyond Minimax Optimality: A Subgame Perfect Gradient Method},
  journal = {Mathematical Programming},
  year    = {2026},
  doi     = {10.1007/s10107-026-02362-0},
  url     = {https://doi.org/10.1007/s10107-026-02362-0}
}

@article{KF16,
  author  = {Kim, Donghwan and Fessler, Jeffrey A.},
  title   = {Optimized First-Order Methods for Smooth Convex Minimization},
  journal = {Mathematical Programming},
  volume  = {159},
  number  = {1--2},
  pages   = {81--107},
  year    = {2016},
  doi     = {10.1007/s10107-015-0949-3},
  url     = {https://doi.org/10.1007/s10107-015-0949-3}
}

@article{KF21,
  author  = {Kim, Donghwan and Fessler, Jeffrey A.},
  title   = {Optimizing the Efficiency of First-Order Methods for Decreasing the Gradient of Smooth Convex Functions},
  journal = {Journal of Optimization Theory and Applications},
  volume  = {188},
  number  = {1},
  pages   = {192--219},
  year    = {2021},
  doi     = {10.1007/s10957-020-01770-2},
  url     = {https://doi.org/10.1007/s10957-020-01770-2}
}

@inproceedings{KOPR23,
  author    = {Kim, Jaeyeon and Ozdaglar, Asuman and Park, Chanwoo and Ryu, Ernest K.},
  title     = {Time-Reversed Dissipation Induces Duality Between Minimizing Gradient Norm and Function Value},
  booktitle = {Advances in Neural Information Processing Systems},
  volume    = {36},
  pages     = {23389--23440},
  year      = {2023},
  url       = {https://proceedings.neurips.cc/paper_files/paper/2023/hash/4947292b9f5e7d4ab792fa35537f8b96-Abstract-Conference.html}
}

@article{KPDOR24,
  author        = {Kim, Jaeyeon and Park, Chanwoo and Ozdaglar, Asuman and Diakonikolas, Jelena and Ryu, Ernest K.},
  title         = {Mirror Duality in Convex Optimization},
  journal       = {arXiv preprint arXiv:2311.17296},
  year          = {2023},
  eprint        = {2311.17296},
  archivePrefix = {arXiv},
  primaryClass  = {math.OC},
  doi           = {10.48550/arXiv.2311.17296},
  url           = {https://arxiv.org/abs/2311.17296}
}

@article{KimRyuDasGupta26,
  author        = {Kim, Heechang and Ryu, Ernest K. and Das Gupta, Shuvomoy},
  title         = {A Domain-Specific Harness for End-to-End Automation of Optimization Research},
  journal       = {arXiv preprint arXiv:2608.07407},
  year          = {2026},
  eprint        = {2608.07407},
  archivePrefix = {arXiv},
  primaryClass  = {math.OC},
  doi           = {10.48550/arXiv.2608.07407},
  url           = {https://arxiv.org/abs/2608.07407}
}

@article{Shu2026Hduality,
  author        = {Shu, Kevin and Wang, Alex L.},
  title         = {A Unified Theory of {H}-Duality in First-Order Methods},
  journal       = {arXiv preprint arXiv:2609.03281},
  year          = {2026},
  eprint        = {2609.03281},
  archivePrefix = {arXiv},
  primaryClass  = {math.OC},
  doi           = {10.48550/arXiv.2609.03281},
  url           = {https://arxiv.org/abs/2609.03281}
}

@article{THG17,
  author  = {Taylor, Adrien B. and Hendrickx, Julien M. and Glineur, Fran{\c{c}}ois},
  title   = {Smooth Strongly Convex Interpolation and Exact Worst-Case Performance of First-Order Methods},
  journal = {Mathematical Programming},
  volume  = {161},
  number  = {1--2},
  pages   = {307--345},
  year    = {2017},
  doi     = {10.1007/s10107-016-1009-3},
  url     = {https://doi.org/10.1007/s10107-016-1009-3}
}

@article{YG26,
  author        = {Yoon, TaeHo and Grimmer, Benjamin},
  title         = {A Theory of Composition and Duality of Extremal Optimal Fixed-Point Algorithms},
  journal       = {arXiv preprint arXiv:2605.02231},
  year          = {2026},
  eprint        = {2605.02231},
  archivePrefix = {arXiv},
  primaryClass  = {math.OC},
  doi           = {10.48550/arXiv.2605.02231},
  url           = {https://arxiv.org/abs/2605.02231}
}

@inproceedings{YKSR24,
  author    = {Yoon, TaeHo and Kim, Jaeyeon and Suh, Jaewook J. and Ryu, Ernest K.},
  title     = {Optimal Acceleration for Minimax and Fixed-Point Problems Is Not Unique},
  booktitle = {Proceedings of the 41st International Conference on Machine Learning},
  series    = {Proceedings of Machine Learning Research},
  volume    = {235},
  pages     = {57244--57314},
  publisher = {PMLR},
  year      = {2024},
  url       = {https://proceedings.mlr.press/v235/yoon24b.html}
}

@article{YRG25,
  author        = {Yoon, TaeHo and Ryu, Ernest K. and Grimmer, Benjamin},
  title         = {{H}-Invariance Theory: A Complete Characterization of Minimax Optimal Fixed-Point Algorithms},
  journal       = {Mathematical Programming},
  year          = {2026},
  note          = {Accepted for publication},
  eprint        = {2511.14915},
  archivePrefix = {arXiv},
  primaryClass  = {math.OC},
  doi           = {10.48550/arXiv.2511.14915},
  url           = {https://arxiv.org/abs/2511.14915}
}

@article{ZG26,
  author        = {Zoll, Aaron and Grimmer, Benjamin},
  title         = {A Complete Characterization of Optimal Subgradient Methods for {L}ipschitz Convex Minimization},
  journal       = {arXiv preprint arXiv:2607.19240},
  year          = {2026},
  eprint        = {2607.19240},
  archivePrefix = {arXiv},
  primaryClass  = {math.OC},
  doi           = {10.48550/arXiv.2607.19240},
  url           = {https://arxiv.org/abs/2607.19240}
}

@article{beck2009fast,
  author  = {Beck, Amir and Teboulle, Marc},
  title   = {A Fast Iterative Shrinkage-Thresholding Algorithm for Linear Inverse Problems},
  journal = {SIAM Journal on Imaging Sciences},
  volume  = {2},
  number  = {1},
  pages   = {183--202},
  year    = {2009},
  doi     = {10.1137/080716542},
  url     = {https://doi.org/10.1137/080716542}
}

@article{d2021acceleration,
  author  = {d'Aspremont, Alexandre and Scieur, Damien and Taylor, Adrien B.},
  title   = {Acceleration Methods},
  journal = {Foundations and Trends in Optimization},
  volume  = {5},
  number  = {1--2},
  pages   = {1--245},
  year    = {2021},
  doi     = {10.1561/2400000036},
  url     = {https://doi.org/10.1561/2400000036}
}

@article{jang2025computer,
  author  = {Jang, Uijeong and Das Gupta, Shuvomoy and Ryu, Ernest K.},
  title   = {Computer-Assisted Design of Accelerated Composite Optimization Methods: {OptISTA}},
  journal = {Mathematical Programming},
  year    = {2025},
  doi     = {10.1007/s10107-025-02258-5},
  url     = {https://doi.org/10.1007/s10107-025-02258-5}
}

@inproceedings{lee2021geometric,
  author    = {Lee, Jongmin and Park, Chanwoo and Ryu, Ernest K.},
  title     = {A Geometric Structure of Acceleration and Its Role in Making Gradients Small Fast},
  booktitle = {Advances in Neural Information Processing Systems},
  volume    = {34},
  pages     = {11999--12012},
  year      = {2021},
  url       = {https://proceedings.neurips.cc/paper/2021/hash/647c722bf90a49140184672e0d3723e3-Abstract.html}
}

@book{nemirovsky1983problem,
  author     = {Nemirovsky, Arkadi S. and Yudin, David B.},
  title      = {Problem Complexity and Method Efficiency in Optimization},
  translator = {Dawson, E. R.},
  series     = {Wiley-Interscience Series in Discrete Mathematics},
  publisher  = {Wiley},
  address    = {Chichester and New York},
  year       = {1983},
  isbn       = {978-0-471-10345-5}
}

@article{nesterov2012make,
  author  = {Nesterov, Yurii},
  title   = {How to Make the Gradients Small},
  journal = {OPTIMA: Mathematical Optimization Society Newsletter},
  volume  = {88},
  pages   = {10--11},
  year    = {2012},
  url     = {https://www.mathopt.org/Optima-Issues/optima88.pdf}
}

@article{park2021factor,
  author  = {Park, Chanwoo and Park, Jisun and Ryu, Ernest K.},
  title   = {Factor-$\sqrt{2}$ Acceleration of Accelerated Gradient Methods},
  journal = {Applied Mathematics \& Optimization},
  volume  = {88},
  number  = {3},
  pages   = {77},
  year    = {2023},
  doi     = {10.1007/s00245-023-10047-9},
  url     = {https://doi.org/10.1007/s00245-023-10047-9}
}

@article{pepit2022,
  author  = {Goujaud, Baptiste and Moucer, C{\'e}line and Glineur, Fran{\c{c}}ois and Hendrickx, Julien M. and Taylor, Adrien B. and Dieuleveut, Aymeric},
  title   = {{PEPit}: Computer-Assisted Worst-Case Analyses of First-Order Optimization Methods in {P}ython},
  journal = {Mathematical Programming Computation},
  volume  = {16},
  pages   = {337--367},
  year    = {2024},
  doi     = {10.1007/s12532-024-00259-7},
  url     = {https://doi.org/10.1007/s12532-024-00259-7}
}

@book{rockafellar1970,
  author    = {Rockafellar, R. Tyrrell},
  title     = {Convex Analysis},
  series    = {Princeton Mathematical Series},
  volume    = {28},
  publisher = {Princeton University Press},
  address   = {Princeton, NJ},
  year      = {1970},
  isbn      = {978-0-691-08069-7},
  doi       = {10.1515/9781400873173},
  url       = {https://doi.org/10.1515/9781400873173}
}

@article{taylor2017composite,
  author  = {Taylor, Adrien B. and Hendrickx, Julien M. and Glineur, Fran{\c{c}}ois},
  title   = {Exact Worst-Case Performance of First-Order Methods for Composite Convex Optimization},
  journal = {SIAM Journal on Optimization},
  volume  = {27},
  number  = {3},
  pages   = {1283--1313},
  year    = {2017},
  doi     = {10.1137/16M108104X},
  url     = {https://doi.org/10.1137/16M108104X}
}

@article{taylordrori2022,
  author  = {Taylor, Adrien B. and Drori, Yoel},
  title   = {An Optimal Gradient Method for Smooth Strongly Convex Minimization},
  journal = {Mathematical Programming},
  volume  = {199},
  number  = {1--2},
  pages   = {557--594},
  year    = {2023},
  doi     = {10.1007/s10107-022-01839-y},
  url     = {https://doi.org/10.1007/s10107-022-01839-y}
}

@article{upadhyaya2026optimalfirstordermethodsmooth,
  author        = {Upadhyaya, Manu and Thomsen, Daniel Berg and Dieuleveut, Aymeric and Taylor, Adrien B.},
  title         = {An Optimal First-Order Method for Smooth and Strongly Convex Composite Optimization and Its Stationary Limit},
  journal       = {arXiv preprint arXiv:2605.22929},
  year          = {2026},
  eprint        = {2605.22929},
  archivePrefix = {arXiv},
  primaryClass  = {math.OC},
  doi           = {10.48550/arXiv.2605.22929},
  url           = {https://arxiv.org/abs/2605.22929}
}
}
\clearpage
\appendix
\section{Deferred Derivations of Dual PEPs}

\subsection{Proof of Proposition~\ref{prop:distance-dual}}
\label{app:distance-dual}
Consider the dual formulation~\eqref{eq:PEP-SDP-dual-with-identity} and
write $\Lambda=(\lambda_{i,j})_{i,j\in\mathcal I}$. Agreement of the constant
terms in the affine identity~\eqref{eq:main-identity} gives $\sigma=r$. Next we match the coefficients of
$f_i-f_\star$ between the two sides. The left-hand side has coefficient zero if $i<N$ and coefficient $-1$ if $i=N$. The right-hand side only has $f_i-f_\star$ occur in $\mathcal{Q}_{i,j}$ and $\mathcal{Q}_{j,i}$ terms, with coefficient $1$ and $-1$ respectively. Hence agreement of these coefficients requires for $i<N$ and $i=N$ that
\begin{equation} \label{eq:f_i-coefficient-cond}
    0 = \sum_{j\in\mathcal{I}_\star} (\lambda_{i,j} - \lambda_{j,i}), \qquad -1 = \sum_{j\in\mathcal{I}_\star} (\lambda_{N,j} - \lambda_{j,N}).
\end{equation}
Combining this with the diagonal convention~\eqref{eq:distance-lambda-diagonal} gives
\[
    (\lambda_{\star,i})_{i\in\mathcal I}
    =-\Lambda^\top\mathbf1,
    \qquad
    (\lambda_{i,\star})_{i\in\mathcal I}
    =-\Lambda\mathbf1-e_N.
\]
Thus nonnegativity of the multipliers involving $\star$ is equivalent to
\[
    \Lambda^\top\mathbf1\leq0,
    \qquad
    \Lambda\mathbf1\leq-e_N.
\]
Finally, we consider the Gram coefficients in the identity~\eqref{eq:main-identity}. The left-hand side has no dependence on $G$, so we require the right-hand side vanish. In particular, recalling our convention that $W_{i,i}=1$ and $L=1$, the right-hand side can be expanded to
\begin{align*}
    0 =&\sum_{i,j\in\mathcal{I}} \lambda_{i,j}\left(-\sum_{k=0}^N (W_{j,k} - W_{i,k})\langle g_j, g_k\rangle - \frac{1}{2}\|g_i\|^2 + \frac{1}{2}\|g_j\|^2\right) \\
    &+\sum_{i\in\mathcal{I}} \lambda_{\star,i}\left(-\sum_{k=0}^{i} W_{i,k}\langle g_i,g_k\rangle + \langle g_i, x_0-x_\star\rangle + \frac{1}{2}\|g_i\|^2\right) \\
    &+\sum_{i\in\mathcal{I}} \lambda_{i,\star}\left(- \frac{1}{2}\|g_i\|^2\right) \\
    &-\frac{\sigma}{2}\|x_0-x_\star\|^2 \\
    &+\langle Z,G \rangle
\end{align*}
with each line above corresponding to the Gram terms in $\mathcal{Q}_{i,j},\mathcal{Q}_{\star,i},\mathcal{Q}_{i,\star},\mathcal{D},$ and the matrix inner product.
From the conditions~\eqref{eq:f_i-coefficient-cond}, the explicitly separated $\|g_i\|^2$ terms cancel for $i<N$, while those involving $\|g_N\|^2$ have total coefficient $+1/2$. Rearranging for $\langle Z,G\rangle$ changes this coefficient to $-1/2$ and gives
\begin{align*}
    \langle Z,G\rangle = &-\frac{1}{2}\|g_N\|^2 + \sum_{i,j\in\mathcal{I}} \lambda_{i,j}\left(\sum_{k=0}^N (W_{j,k} - W_{i,k})\langle g_j, g_k\rangle\right) \\
    &+\sum_{i\in\mathcal{I}} \lambda_{\star,i}\left(\sum_{k=0}^{i} W_{i,k}\langle g_i,g_k\rangle - \langle g_i, x_0-x_\star\rangle\right) + \frac{\sigma}{2}\|x_0-x_\star\|^2\\
    &= \left\langle \begin{pmatrix}
            \frac{\sigma}{2} & \frac{1}{2}\mathbf{1}^\top\Lambda\\
            \frac{1}{2}\Lambda^\top \mathbf{1} & -\operatorname{sym}\left(\Lambda^\top W\right)-\frac{1}{2}e_Ne_N^\top
    \end{pmatrix}, G\right\rangle.
\end{align*}
So any feasible solution must have
\[
Z=
\begin{pmatrix}
    \frac r2 & \frac12\mathbf1^\top\Lambda\\[1mm]
    \frac12\Lambda^\top\mathbf1
    &-\operatorname{sym}(\Lambda^\top W)-\frac12e_Ne_N^\top
\end{pmatrix}.
\]
Consequently, the identity constraint~\eqref{eq:PEP-SDP-dual-with-identity} allows us to eliminate $\sigma,\lambda_{i,\star},\lambda_{\star,i},Z$, resulting in the claimed, simplified dual statement~\eqref{eq:PEP-SDP-dual-clean}. Positive semidefiniteness of the displayed
block gives $r\geq0$. Conversely, from any feasible $(\Lambda,r)$ in
\eqref{eq:PEP-SDP-dual-clean}, define $\sigma=r$, recover the
multipliers involving $\star$ from the preceding display, and take $Z$ to be
the displayed block matrix. These choices recover~\eqref{eq:main-identity}
coefficient by coefficient.

\subsection{Proof of Proposition~\ref{prop:gradient-dual}}
\label{app:gradient-dual}
Consider~\eqref{eq:ogmg-main-identity} and write
$\Lambda=(\lambda_{i,j})_{i,j\in\mathcal I}$. Agreement of the constant terms
gives $\sigma=r$. Neither the objective nor the normalization contains
$\|x_0-x_\star\|^2$, so the corresponding diagonal entry of $Z$ must vanish.
Positive semidefiniteness then forces the entire corresponding row and column
to vanish. Denote the remaining bottom-right block of $Z$ by $Z_g$. Matching the coefficients of
$\langle g_i,x_0-x_\star\rangle$ therefore gives
\[
    \lambda_{\star,i}=0
    \qquad \forall i\in\mathcal I.
\]
Matching the function-value coefficients and using
\eqref{eq:gradient-lambda-diagonal} then gives
\[
    (\lambda_{i,\star})_{i\in\mathcal I}
    =-\Lambda\mathbf1,
    \qquad
    \Lambda^\top\mathbf1=-r e_0.
\]
Thus nonnegativity of the remaining multipliers involving $\star$ is
equivalent to $\Lambda\mathbf1\leq0$.
Collecting the gradient Gram coefficients in the identity~\eqref{eq:ogmg-main-identity}, following the same process as Proposition~\ref{prop:distance-dual}, gives
\[
Z_g=
-\operatorname{sym}(\Lambda^\top W)
-\frac r2e_0e_0^\top
-\frac12e_Ne_N^\top.
\]
This proves that the identity formulation implies
\eqref{eq:ogmg-PEP-SDP-dual-clean}. Conversely, suppose
$(\Lambda,r)$ is feasible in~\eqref{eq:ogmg-PEP-SDP-dual-clean}. Since
\[
    -r=\mathbf1^\top\Lambda\mathbf1\leq0,
\]
one has $r\geq0$. Define $\sigma=r$, set
$\lambda_{\star,i}=0$ and
$\lambda_{i,\star}=-(\Lambda\mathbf1)_i$, and take
$Z$ as the matrix with zero first row and column, setting the remaining block as $Z_g$. These choices recover~\eqref{eq:ogmg-main-identity}
coefficient by coefficient.

\subsection{Proof of Proposition~\ref{prop:lemni-dual}}
\label{app:lemni-dual}
Consider~\eqref{eq:lemni-main-identity} and write
$\Lambda=(\lambda_{i,j})_{i,j\in\mathcal I}$. Agreement of the constant terms
gives $\sigma=r$. Matching the function-value coefficients and using
\eqref{eq:lemni-lambda-diagonal} gives
\[
    (\lambda_{\star,i})_{i\in\mathcal I}
    =-\Lambda^\top\mathbf1,
    \qquad
    (\lambda_{i,\star})_{i\in\mathcal I}
    =-\Lambda\mathbf1.
\]
Thus nonnegativity of the multipliers involving $\star$ is equivalent to
\[
    \Lambda^\top\mathbf1\leq0,
    \qquad
    \Lambda\mathbf1\leq0.
\]
Collecting the Gram coefficients in the identity~\eqref{eq:lemni-main-identity}, following the same process as Proposition~\ref{prop:distance-dual}, gives
\[
Z=
\begin{pmatrix}
    \frac r2 & \frac12\mathbf1^\top\Lambda\\[1mm]
    \frac12\Lambda^\top\mathbf1
    &-\operatorname{sym}(\Lambda^\top W)-\frac12e_Ne_N^\top
\end{pmatrix}.
\]
This proves that the identity formulation implies
\eqref{eq:lemni-PEP-SDP-dual-clean}. Positive semidefiniteness of the
displayed block gives $r\geq0$. Conversely, from any feasible $(\Lambda,r)$
in~\eqref{eq:lemni-PEP-SDP-dual-clean}, define $\sigma=r$,
recover the multipliers involving $\star$ from the preceding display, and
take $Z$ to be the displayed block matrix. These choices recover
\eqref{eq:lemni-main-identity} coefficient by coefficient.

\section{Verification of Strong Duality Claims}
\subsection{A common coercive combination}

\begin{lemma}\label{lem:app-coercive-combination}
Fix a method by its unit lower triangular matrix $W$. For arbitrary vectors
$g_0,\ldots,g_N$, set
\[
 d_i=g_i-g_0,
 \qquad
 h_i=\sum_{j=0}^{i-1}W_{i,j}g_j
 \quad \forall 1\le i\le N.
\]
There are constants $a,c_1,\ldots,c_N>0$, depending only on $W$, such that
\begin{equation}\label{eq:app-positive-form}
 \mathcal B_W(g):=
 a\|g_0\|^2+
 \sum_{i=1}^Nc_i\bigl(\|d_i\|^2+\langle d_i,h_i\rangle\bigr)
 \ge \|g_0\|^2+\frac12\sum_{i=1}^N\|d_i\|^2.
\end{equation}
\end{lemma}

\begin{proof}
Write $\alpha_i=\sum_{j<i}W_{i,j}$ and
$\kappa_i=\alpha_i^2+\sum_{j=1}^{i-1}W_{i,j}^2$. Since
$h_i=\alpha_i g_0+\sum_{j=1}^{i-1}W_{i,j}d_j$, Cauchy--Schwarz gives
\[
 \|h_i\|^2\le
 \kappa_i\left(\|g_0\|^2+
 \sum_{j=1}^{i-1}\|d_j\|^2\right).
\]
Choose $c_N=1$ and, backward for $j=N-1,\ldots,1$,
\[
 c_j=1+\sum_{i=j+1}^Nc_i\kappa_i,
 \qquad
 a=1+\frac12\sum_{i=1}^Nc_i\kappa_i.
\]
Young's inequality and the preceding bound imply
\[
 \|d_i\|^2+\langle d_i,h_i\rangle
 \ge\frac12\|d_i\|^2-\frac12\|h_i\|^2.
\]
After summing, the coefficient of $\|g_0\|^2$ is
$a-\frac12\sum_i c_i\kappa_i=1$, while that of $\|d_j\|^2$ is
$\frac12(c_j-\sum_{i>j}c_i\kappa_i)=1/2$. This is
\eqref{eq:app-positive-form}.
\end{proof}

We use the standard partial-Slater theorem for finite-dimensional conic
programs~\cite[Theorem~28.2]{rockafellar1970}. Because all constraints other
than the positive semidefinite constraint are polyhedral, it is enough to
construct a feasible dual point whose semidefinite slack lies in the relative
interior of its PSD face.

\subsection{Proof of Proposition~\ref{prop:distance-duality}}
\label{app:distance-duality}
Let $v=x_0-x_\star$. The primal~\eqref{eq:PEP-SDP} is feasible
(take the zero trace) and has finite objective. Indeed,
$\|v\|\leq\sqrt2$, $\|g_i\|\leq\|x_i-x_\star\|$, and the fixed-step relations
bound the iterates and gradients inductively; smoothness then gives
$f_N-f_\star\leq\|x_N-x_\star\|^2/2$.

Choose $a,c_i$ from Lemma~\ref{lem:app-coercive-combination}. For $t>0$, set
\[
\Lambda_t
=-ta\,e_0e_0^\top
-t\sum_{i=1}^Nc_i(e_i-e_0)(e_i-e_0)^\top
-e_Ne_N^\top.
\]
Its off-diagonal entries are nonnegative and $\Lambda_t\mathbf1
    =\Lambda_t^\top\mathbf1
    =-ta\,e_0-e_N$,
so all scalar constraints in~\eqref{eq:PEP-SDP-dual-clean} hold at any level
$R>0$. The quadratic form represented by its semidefinite block is
\begin{equation}\label{eq:app-distance-slack}
\begin{aligned}
&t\mathcal B_W(g)+\langle g_N,h_N\rangle
+\frac12\|g_N\|^2\\
&\qquad
-\langle ta g_0+g_N,v\rangle
+\frac R2\|v\|^2.
\end{aligned}
\end{equation}
The coordinates $(g_0,d_1,\ldots,d_N)$ are related invertibly to
$(g_0,\ldots,g_N)$, and~\eqref{eq:app-positive-form} is coercive in the
former. Hence, for all sufficiently large $t$, the first line of
\eqref{eq:app-distance-slack} is positive definite in the gradients. A Schur
complement then makes the whole form positive definite after increasing $R$.
Thus the compact dual has a feasible point with a positive-definite
semidefinite block. Partial Slater gives equality of the primal and dual
values.

\subsection{Proof of Proposition~\ref{prop:gradient-duality}}
\label{app:gradient-duality}
The zero trace is feasible for~\eqref{eq:ogmg-PEP-SDP}. Its
objective is finite: smooth convexity gives
$\|g_0\|^2\leq2(f_0-f_\star)\leq2$, and $\|g_i-g_0\|\leq\|x_i-x_0\|
 \leq\sum_{j<i}|W_{i,j}|\,\|g_j\|$
bounds the remaining gradients inductively.

Choose $a,c_i$ from Lemma~\ref{lem:app-coercive-combination}. For $t>0$, set
$R=2ta$ and
\[
\Lambda_t
=-R\,e_0e_0^\top
-t\sum_{i=1}^Nc_i(e_i-e_0)(e_i-e_0)^\top.
\]
Then $\Lambda_t\mathbf1=-R e_0$ and $\Lambda_t^\top\mathbf1=-R e_0$.
Moreover, the off-diagonal entries of $\Lambda$ are nonnegative, so all scalar constraints in
\eqref{eq:ogmg-PEP-SDP-dual-clean} hold at level $R$. The quadratic form
represented by its semidefinite constraint is
\[
    t\mathcal B_W(g)-\frac12\|g_N\|^2.
\]
By~\eqref{eq:app-positive-form}, this form is positive definite for
sufficiently large $t$. Thus the compact dual has a feasible point with a positive-definite
semidefinite block. Partial Slater gives equality of the primal and dual
values.

\subsection{Proof of Proposition~\ref{prop:lemni-duality}}
\label{app:lemni-duality}
Let $v=x_0-x_\star$. The zero trace is feasible for
\eqref{eq:lemni-PEP-SDP}, and its objective is finite. Indeed,
$\|v\|\leq\sqrt2$, smooth convexity gives $\|g_0\|\leq\|v\|$, and the
fixed-step relations bound the remaining iterates and gradients inductively.

Choose $a,c_i$ from Lemma~\ref{lem:app-coercive-combination}. For $t>0$, set
\[
\Lambda_t
=-ta\,e_0e_0^\top
-t\sum_{i=1}^Nc_i(e_i-e_0)(e_i-e_0)^\top.
\]
Its off-diagonal entries are nonnegative and $\Lambda_t\mathbf1
    =\Lambda_t^\top\mathbf1
    =-ta\,e_0$,
so all scalar constraints in~\eqref{eq:lemni-PEP-SDP-dual-clean} hold at any
level $R>0$. The quadratic form represented by its semidefinite block is
\begin{equation}\label{eq:app-lemni-slack}
    t\mathcal B_W(g)
    -ta\langle g_0,v\rangle
    +\frac R2\|v\|^2
    -\frac12\|g_N\|^2.
\end{equation}
By~\eqref{eq:app-positive-form}, the gradient part is positive definite for
all sufficiently large $t$. A Schur complement then makes
\eqref{eq:app-lemni-slack} positive definite after increasing $R$. Thus the
compact dual has a feasible point with a positive-definite semidefinite block.
Partial Slater gives equality of the primal and dual values.

\section{Deferred Proofs of Matrix Factorizations and Bounds} \label{app:LA}

\subsection{Proof of Proposition~\ref{prop:M-matrix-factorization}}
We argue by induction on the dimension. Write
\[
A=\begin{pmatrix}A_0&u\\v^\top&a\end{pmatrix}.
\]
The row-sum condition gives $a\geq0$. If $a=0$, the nonnegative last row and column sums, together with the nonpositive off-diagonal entries, force $u=v=0$. We then factor $A_0$ by induction and append a zero diagonal entry to $T$ and a unit diagonal entry to $C$.

Suppose $a>0$, and set $B=A_0-a^{-1}uv^\top$. It has nonpositive off-diagonal entries. Writing $\rho=A\mathbf1$ and $\kappa=A^\top\mathbf1$, one has
\[
B\mathbf1=\rho_{0:N-1}-\frac{\rho_N}{a}u\geq0,
\qquad
\mathbf1^\top B=\kappa_{0:N-1}^\top-\frac{\kappa_N}{a}v^\top\geq0.
\]
Factor $B=T_0C_0$ by induction and set
\[
T=\begin{pmatrix}T_0&u\\0&a\end{pmatrix},
\qquad
C=\begin{pmatrix}C_0&0\\a^{-1}v^\top&1\end{pmatrix}.
\]
Then $A=TC$ and the stated triangular signs hold. Since $C$ is unit lower triangular with nonpositive entries below the diagonal,
\[
C^{-1}=I+(I-C)+(I-C)^2+\cdots\geq0,
\]
and hence
\[
\mathbf1^\top T=\mathbf1^\top AC^{-1}\geq0.
\]
If $A$ is nonsingular, every diagonal entry of $T$ is positive, the recursion is unique, and the same finite Neumann-series argument gives $T^{-1},C^{-1}\geq0$.

Suppose now that $A\mathbf1=e_0$. When $a>0$, the last row sum is zero, so $v^\top\mathbf1+a=0$, while the displayed formula for $B\mathbf1$ gives $B\mathbf1=e_0$. Thus the inductive factor may be chosen with $C_0\mathbf1=e_0$, and the appended row of $C$ has sum zero. When $a=0$, $u=v=0$, and we may choose the entries below the diagonal in the appended row of $C$ to sum to $-1$. This again gives $C\mathbf1=e_0$. Consequently $Te_0=TC\mathbf1=e_0$, and therefore $T_{0,0}=1$. Finally, if $\mathbf1^\top A=e_0^\top$, then
$\mathbf1^\top T=e_0^\top C^{-1}=e_0^\top$, which also gives $T_{0,0}=1$.

\subsection{Proof of Lemma~\ref{lem:app-inverse-bound}}
Set $H=\operatorname{sym}(A)$. It is a symmetric matrix with nonpositive off-diagonal entries and $H\mathbf1\geq0$. Moreover,
\[
x^\top Hx
=\sum_i(H\mathbf1)_i x_i^2
 +\sum_{i<j}(-H_{i,j})(x_i-x_j)^2
\geq0.
\]
Let the connected components of the graph with edges $H_{i,j}<0$ index a simultaneous block decomposition of $A$. Indeed, if two indices lie in different components, then both corresponding off-diagonal entries of $A$ vanish.

On a connected component $K$, if $(H\mathbf1)_i>0$ for some $i\in K$, the displayed quadratic form is positive definite on that component. Hence $A_{K,K}$ is nonsingular and is a nonsingular $M$-matrix. If $H\mathbf1$ vanishes on $K$, then the nonnegative row and column sums of $A$ both vanish there, so
\[
A_{K,K}\mathbf1=0,
\qquad
\mathbf1^\top A_{K,K}=0.
\]
The same quadratic-form identity shows that the right and left kernels of this block are both spanned by $\mathbf1$. Thus $A$ is a direct sum of nonsingular $M$-matrix blocks and singular blocks having zero row and column sums. This also proves
$\operatorname{range}(A)=\operatorname{range}(A^\top)$.

If $Au=e_i$ is solvable, the index $i$ cannot belong to a singular block, since left multiplication by the all-ones vector of that block would give $0=1$. Thus the coordinates of $u$ on every nonsingular block are uniquely determined, and all nonsingular blocks other than the one containing $i$ contribute zero. The same applies to $v$ and $j$.

It remains only to use the usual inverse-entry comparison on a nonsingular block. Its inverse is entrywise nonnegative. Fix a column $j$ of the inverse. If a largest entry were strictly larger than the diagonal entry in that column, the row-sum expansion over the indices attaining the largest value would force a nonzero null vector in the corresponding principal submatrix, a contradiction. Hence every entry in column $j$ is at most the $j$th diagonal entry. Applying the same argument to the transpose bounds it by the corresponding row diagonal entry. This gives the displayed inequality. Every principal submatrix satisfies the same off-diagonal sign and nonnegative row- and column-sum assumptions, so the argument applies verbatim.

\section{Uniqueness of Limiting Optimal Factors} \label{app:limiting-factors}

\subsection{Proof of Proposition~\ref{prop:ogm-limiting-factors}}
For each $k$, let $b^{(k)}$ and $\delta^{(k)}$ be the vectors constructed in
Lemma~\ref{lem:ogm-TC-projection}. Since $r_\mathtt{OGM}
    \leq\Phi(\delta^{(k)})
    \leq r^{(k)}$,
compactness of $\Delta_{N+1}$ and uniqueness in
Lemma~\ref{lem:ogm-simplex-prob} give
$\delta^{(k)}\to\delta^\star$. For each $i<N$, every quantity in the
chain proving~\eqref{eq:ogm-projected-chain} therefore converges to
$r_\mathtt{OGM}$. Hence
\[
    b_i^{(k)}\to1,
    \qquad
    \frac{T_{i,i}^{(k)}b_i^{(k)}}
    {\sum_{j=0}^i\delta_j^{(k)}}\to1,
\]
and consequently $T_{i,i}^{(k)}
    \to\sum_{j=0}^i\delta_j^\star
    =(T_F^\star)_{i,i}$.
The terminal chain similarly gives
\[
    T_{N,N}^{(k)}\to1,
    \qquad
    b_N^{(k)}\to1.
\]
Thus $b^{(k)}\to\mathbf1$. Since $b^{(k)}\leq C^{(k)}\mathbf1\leq\mathbf1$,
we have $C^{(k)}\mathbf1\to\mathbf1$. Moreover, the sign constraints on
$C^{(k)}$ give $C^{(k)}\to I=C_F^\star$.
The gap in~\eqref{eq:ogm-cut-r} also converges to zero. Since
$b_j^{(k)}\to1$, every entry of $T^{(k)}$ above the first superdiagonal
converges to zero. Finally, the equations
$T^{(k)}b^{(k)}=e_N$ give
\[
    T_{i,i+1}^{(k)}\to-(T_F^\star)_{i,i}
    \qquad \forall 0\leq i<N.
\]
Therefore $T^{(k)}\to T_F^\star$.

\subsection{Proof of Proposition~\ref{prop:ogmg-limiting-factors}}
For each $k$, choose $b^{(k)}$ as in~\eqref{eq:ogmg-backward-solution} and
define $q^{(k)},D_i^{(k)}$, and $\delta^{(k)}$ as in
\eqref{eq:ogmg-green-quantities}--\eqref{eq:ogmg-delta-from-TC}. Since $r_\mathtt{OGM}
    \leq \Phi(\delta^{(k)})
    \leq r^{(k)}$,
compactness of $\Delta_{N+1}$ and uniqueness in
Lemma~\ref{lem:ogm-simplex-prob} give
$\delta^{(k)}\to\delta^\star$. Since $q_0^{(k)}=1$, the terminal chain gives
\[
    1\leq D_0^{(k)}
    =\frac{b_0^{(k)}}{\delta_N^{(k)}}
    \leq\frac{\sqrt{r^{(k)}}}{\delta_N^{(k)}}
    \longrightarrow1.
\]
Hence
\[
    b_i^{(k)}q_i^{(k)}
    =D_0^{(k)}\delta_{N-i}^{(k)}
    \longrightarrow\delta_{N-i}^\star>0.
\]
Thus $T_{i,i}^{(k)}>0$ for every $i$ and all sufficiently large $k$, and
\eqref{eq:ogmg-singular-inverse-bound} gives $q_i^{(k)}\leq1$. Therefore
each $b_i^{(k)}$ is bounded away from zero, and
\eqref{eq:ogmg-diagonal-conditions} bounds every diagonal entry of
$T^{(k)}$.

Moreover, $C^{(k)}\mathbf1=e_0$ and the sign constraints give $\sum_{j<i}(-C_{i,j}^{(k)})=1$ for all $1\leq i\leq N$. Hence, $C^{(k)}$ is bounded. Also,
\[
    (T^{(k)})^\top\mathbf1
    =(C^{(k)})^{-\top}
      \bigl((C^{(k)})^\top(T^{(k)})^\top\mathbf1\bigr)
    \geq0,
\]
and hence
\[
    \sum_{i<j}(-T_{i,j}^{(k)})
    \leq T_{j,j}^{(k)}
    \qquad \forall 1\leq j\leq N.
\]
Thus $(T^{(k)},C^{(k)})$ is bounded. Any convergent subsequence has a limit
that is feasible in~\eqref{eq:ogmg-TC-problem} at $r_\mathtt{OGM}$. The
equality case of Lemma~\ref{lem:ogmg-TC-projection} forces this limit to be
$(T_G^\star,C_G^\star)$.

\subsection{Proof of Proposition~\ref{prop:lemni-limiting-factors}}
For each $k$, use the construction in the proof of
Lemma~\ref{lem:lemni-TC-projection} to define
$b^{(k)},E_i^{(k)},D_i^{(k)},R_i^{(k)},\rho_i^{(k)}$, and $m^{(k)}$. The
projected scalar chains satisfy
\[
    \frac1{\Omega_N}
    \leq
    \max_{0\leq i\leq N}
    \frac{(\rho_i^{(k)}-\rho_{i+1}^{(k)})^2}
    {\rho_i^{(k)}(1-(\rho_{i+1}^{(k)})^2)}
    \leq m^{(k)}\leq\sqrt{r^{(k)}}.
\]
Compactness and uniqueness in Lemma~\ref{lem:lemni-scalar-prob} therefore give
\[
    \rho_i^{(k)}\to\rho_i
    \qquad \forall 0\leq i\leq N+1,
    \qquad
    m^{(k)}\to\frac1{\Omega_N}.
\]
Every local term at the limiting scalar sequence equals $1/\Omega_N$, so the
path ratios in~\eqref{eq:lemni-path-inequality} also converge to
$1/\Omega_N$. The radial inequalities then give
$R_i^{(k)}/R_{i+1}^{(k)}\to1$. Since
$R_{N+1}^{(k)}=E_N^{(k)}=1$, it follows that $R_i^{(k)}\to1$ for every $i$.
Hence $E_i^{(k)}$ and $D_i^{(k)}$ converge to the values in
\eqref{eq:lemni-equality-path}.

The derivation of~\eqref{eq:lemni-path-inequality} also gives
\[
\frac{(E_i^{(k)}-E_{i-1}^{(k)})^2
 +(D_i^{(k)}-D_{i+1}^{(k)})^2}
 {2E_i^{(k)}D_i^{(k)}}
\leq
m^{(k)}
\frac{T_{i,i}^{(k)}b_i^{(k)}}{E_i^{(k)}}
\frac{b_i^{(k)}}{D_i^{(k)}}
\leq m^{(k)}.
\]
Both factors in the middle lie in $[0,1]$, while the left- and right-hand
sides converge to $1/\Omega_N$. Therefore
\[
    \frac{b_i^{(k)}}{D_i^{(k)}}\to1,
    \qquad
    \frac{T_{i,i}^{(k)}b_i^{(k)}}{E_i^{(k)}}\to1.
\]
Thus $b_i^{(k)}$ is bounded away from zero and
\[
    T_{i,i}^{(k)}\to(T_L^\star)_{i,i}
    \qquad \forall 0\leq i\leq N.
\]
Since $(T^{(k)})^\top\mathbf1
    =(C^{(k)})^{-\top}
      \bigl((C^{(k)})^\top(T^{(k)})^\top\mathbf1\bigr)
    \geq0$,
the sign constraints give
\[
    \sum_{i<j}(-T_{i,j}^{(k)})
    \leq T_{j,j}^{(k)}
    \qquad \forall 1\leq j\leq N.
\]
Hence $T^{(k)}$ is bounded. Its diagonal entries are positive for
all sufficiently large $k$, so $(T^{(k)})^{-1}\geq0$. Therefore
\[
    C^{(k)}\mathbf1
    =(T^{(k)})^{-1}
      T^{(k)}C^{(k)}\mathbf1
    \geq0.
\]
Together with the unit diagonal and nonpositive entries below the diagonal of
$C^{(k)}$, this shows that $C^{(k)}$ is bounded.

Any convergent subsequence has a limit that gives a feasible tuple in
\eqref{eq:lemni-TC-problem} at $r_\mathtt{Lemni}$. The equality case of
Lemma~\ref{lem:lemni-TC-projection} forces this limit to be
$(T_L^\star,C_L^\star)$.

\section{Discovery of Optimal Methods as Necessary Completions} \label{app:rediscovery}

The equality cases above determine the triangular factors without using the
known method formulas. Solving the corresponding zero-slack completions
recovers the three methods directly and independently verifies that they
attain the lower bounds.

\subsection{A Rediscovery of OGM} \label{subsec:discovering-ogm}

Take $C=C_F^\star$, $T=T_F^\star$, and $r=r_\mathtt{OGM}$. The recurrence
\eqref{eq:ogm-rate-recursion} gives
\[
T_F^\star\one=e_N,
\qquad
\one^\top T_F^\star=(\delta^\star)^\top.
\]
Together with the upper bidiagonal form and signs of $T_F^\star$, these
identities verify all factor constraints in~\eqref{eq:ogm-TC-problem}.
Every diagonal condition in~\eqref{eq:ogm-diagonal-conditions} is tight. Hence
the zero-slack equations in Lemma~\ref{lem:ogm-eliminate-W} require the
completion $W$ to satisfy
\begin{equation}\label{eq:forward-tight-completion}
\bigl((T_F^\star)^\top W\bigr)_{n,i}
=
\begin{cases}
(T_F^\star)_{i,i},&n=i,\\[1mm]
\delta_n^\star\delta_i^\star/r_\mathtt{OGM},&0\leq i<n\leq N,\\[1mm]
0,&n<i.
\end{cases}
\end{equation}
Since $(T_F^\star)_{n-1,n}=-(T_F^\star)_{n-1,n-1}$, the equations for entries below the diagonal are
\[
(T_F^\star)_{n,n}W_{n,i}-(T_F^\star)_{n-1,n-1}W_{n-1,i}
=\frac{\delta_n^\star\delta_i^\star}{r_\mathtt{OGM}}
\qquad \forall 0\leq i<n\leq N.
\]
Starting from $W_{i,i}=1$ and summing this recurrence, using
$\sum_{k=0}^m\delta_k^\star=(T_F^\star)_{mm}$ for every $0\leq m\leq N$, yields
\[
W_{n,i}
=\frac{(T_F^\star)_{i,i}}{(T_F^\star)_{n,n}}
+\frac{\delta_i^\star}{r_\mathtt{OGM}}
\left(1-\frac{(T_F^\star)_{i,i}}{(T_F^\star)_{n,n}}\right).
\]
For $n<N$, the first ratio is $\theta_i^2/\theta_n^2$; for $n=N$, it is
$2\theta_i^2/\theta_N^2$; and
$\delta_i^\star/r_\mathtt{OGM}=2\theta_i$. Thus the completion is exactly
$W_\mathtt{OGM}$ from~\eqref{eq:ogm-W}, and is feasible at
$r_\mathtt{OGM}$.

\subsection{A Rediscovery of OGM-G} \label{subsec:discovering-ogmg}

Let $P$ be the upper bidiagonal matrix with diagonal entries one and
superdiagonal entries minus one, and let $J$ be the coordinate-reversal
permutation. For the OGM equality factor $T_F^\star$ in
\eqref{eq:ogm-T-star-r},
\[
T_F^\star=\operatorname{diag}((T_F^\star)_{0,0},\ldots,(T_F^\star)_{N,N})P,
\qquad
S=P^{-\top}J.
\]
The equality factors in
\eqref{eq:ogmg-T-star-r}--\eqref{eq:ogmg-C-star-r} can be written as
\[
T_G^\star
=PJ\operatorname{diag}\bigl(((T_F^\star)_{0,0})^{-1},\ldots,
((T_F^\star)_{N,N})^{-1}\bigr)JP^{-1},
\qquad
C_G^\star=P^\top.
\]
In addition, $C_G^\star\one=e_0$, while $T_G^\star C_G^\star$ has
nonpositive off-diagonal entries, row sums $e_0$, and nonnegative column
sums. Hence all factor constraints in~\eqref{eq:ogmg-TC-problem} hold.

Set
\[
W_G=SW_\mathtt{OGM}^\top S^{-1}=W_\mathtt{OGM-G}.
\]
Using $JP^\top=PJ$, the preceding definitions give
\[
C_G^\star W_G(T_G^\star)^{-\top}
=J\bigl((T_F^\star)^\top W_\mathtt{OGM}\bigr)^\top J,
\qquad
(T_G^\star)^{-1}e_N=J\delta^\star.
\]
By~\eqref{eq:forward-tight-completion}, the matrix on the left has diagonal
$1/(T_G^\star)_{i,i}$ and, for $i>j$, entries below the diagonal
\[
\frac{(e_i^\top(T_G^\star)^{-1}e_N)
(e_j^\top(T_G^\star)^{-1}e_N)}{r_\mathtt{OGM}}.
\]
Moreover, every inequality in~\eqref{eq:ogmg-diagonal-conditions} is tight. Congruence of the semidefinite constraint by
$(T_G^\star)^{-1}$ and $(T_G^\star)^{-\top}$ gives
\[
\operatorname{sym}\bigl(C_G^\star W_G(T_G^\star)^{-\top}\bigr)
-\frac12e_0e_0^\top
-\frac1{2r_\mathtt{OGM}}
 (T_G^\star)^{-1}e_Ne_N^\top(T_G^\star)^{-\top}
\succeq0.
\]
Since every diagonal condition is tight, this matrix must be zero. Thus the completion is
$W_\mathtt{OGM-G}$.

\subsection{A Rediscovery of Lemniscate Acceleration} \label{subsec:discovering-lemni}

Take $C=C_L^\star$, $T=T_L^\star$, and
$r=1/\Omega_N^2$. The recurrence~\eqref{eq:lemni-rate-recursion} and the
telescoping definition of $\phi_i$ give
\[
    TC\mathbf1=e_N,
    \qquad
    \mathbf1^\top TC
    =(\rho_0-\rho_1,\rho_1-\rho_2,\dots,\rho_N-\rho_{N+1}).
\]
A direct multiplication also shows that the first $N$ rows of $TC$ are upper
bidiagonal with nonpositive superdiagonal and that the entries below the diagonal in
its final row are $-(\phi_{j+1}-\phi_j)$. Hence all factor
constraints in~\eqref{eq:lemni-TC-problem} hold. Substitution in
\eqref{eq:lemni-diagonal-conditions} shows that every condition is tight. Using \eqref{eq:lemni-rate-recursion} and the identity
$\rho_j=(1-\rho_{N+1-j})/(1+\rho_{N+1-j})$, the zero-slack equations for
entries below the diagonal in Lemma~\ref{lem:lemni-eliminate-W}
reduce to the recurrence
\[
    W_{n,i}-W_{n-1,i}
    =\Omega_N(\phi_{i+1}-\phi_i)
      (\phi_{N-n+1}-\phi_{N-n})
    \qquad \forall 0\leq i<n\leq N.
\]
Starting from $W_{i,i}=1$ and summing this recurrence gives
\[
    W_{n,i}
    =1+\Omega_N(\phi_{i+1}-\phi_i)
      (\phi_{N-i}-\phi_{N-n})
    \qquad \forall 0\leq i<n\leq N.
\]
Thus the completion is exactly $W_\mathtt{Lemni}$ from
\eqref{eq:lemni-W}.

\section{Clarification of AI assistance}

The authors used AI tools during the preparation of this manuscript. Codex was used to assist with editing, reorganization, consistency checks, and \LaTeX{} maintenance of the draft. In an early exploratory stage, ChatGPT Pro suggested we reformulate using the algebraic decompositions involving upper and lower triangular factors of the $TC$ form. This suggestion was central to our progress here. In its original development, this paper only addressed OGM and OGM-G. Once the paper~\cite{KimRyuDasGupta26} providing Lemniscate acceleration was released (approximately one month before this paper's first public posting), we provided ChatGPT Pro with our working manuscript to extend it to cover Lemniscate. It accomplished this directly, up to copyediting to match the form of our developments in Sections~\ref{sec:ogm} and~\ref{sec:ogmg}.
All mathematical statements, proofs, references, and final editorial decisions in this manuscript are the responsibility of the authors.

\end{document}